\documentclass[a4paper,reqno,11pt]{amsart}
\usepackage[a4paper,top=1in, bottom=1in, left=1in, right=1in]{geometry}
\usepackage{enumerate}
\usepackage{amsmath,amsfonts,amssymb,amsthm}
\usepackage{mathrsfs}
\usepackage{stmaryrd}
\usepackage{xcolor} 
\usepackage[linkcolor=blue,citecolor=cyan,colorlinks]{hyperref}

\allowdisplaybreaks

\newtheorem{theorem}{Theorem}[section]
\newtheorem{prop}[theorem]{Proposition}

\theoremstyle{definition}
\newtheorem{definition}[theorem]{Definition}
\theoremstyle{remark}
\newtheorem*{remark}{Remark}
\numberwithin{equation}{section}

\def\leq{\leqslant}
\def\geq{\geqslant}

\begin{document}
		
		\title[GWP of FBP for incompressible MHD]{Global well-posedness of the free boundary problem for incompressible viscous resistive MHD in critical Besov spaces}
		\author{Wei Zhang}
		
\address{Wei Zhang\newline	
	School of  Artificial Intelligence and Big Data, Hefei University, Hefei,  Anhui, 230601, China}
\email{zhangwei16@mails.ucas.edu.cn}
		
		\author{Jie Fu}
			
		\author{Chengchun Hao}
		
		\author{Siqi Yang}
		
		\address{Jie Fu, Chengchun Hao and Siqi Yang\newline
			HLM, Institute of Mathematics, Academy of Mathematics and Systems Science, Chinese Academy of Sciences, Beijing 100190, China\newline
		\and
			School of Mathematical Sciences,
			University of Chinese Academy of Sciences, 
		Beijing 100049, China}
		\email{fujie@amss.ac.cn,hcc@amss.ac.cn,yangsiqi@amss.ac.cn}	
		 
\begin{abstract}			
This paper aims to establish the global well-posedness of the free boundary problem for the incompressible viscous resistive magnetohydrodynamic (MHD) equations. Under the framework of Lagrangian coordinates, a unique global solution exists in the half-space provided that the norm of the initial data in the critical homogeneous Besov space $\dot{B}_{p, 1}^{-1+N/p}(\mathbb{R}_{+}^N)$ is sufficiently small, where $p \in [N, 2N-1)$. Building upon prior work such as (Danchin and Mucha, J. Funct. Anal. 256 (2009) 881--927) and (Ogawa and Shimizu, J. Differ. Equations 274 (2021) 613--651) in the half-space setting, we establish maximal $L^{1}$-regularity for both the Stokes equations without surface stress and the linearized equations of the magnetic field with zero boundary condition. The existence and uniqueness of solutions to the nonlinear problems are proven using the Banach contraction mapping principle.
\end{abstract}

					

\maketitle

	\section{Introduction}
	In this paper, we address the free boundary problem for the incompressible viscous and resistive magnetohydrodynamic (MHD) equations in $N$-dimensions:
	\begin{align}\label{1.1}
		\left\{\begin{aligned}&\partial_t v+(v \cdot \nabla) v-\operatorname{div}( \mathbf{D}(v)-p \mathbf{I})=(H \cdot \nabla) H-\frac{1}{2} \nabla|H|^2, & \text { in } \Omega(t), \\
			&\partial_t H+(v \cdot \nabla) H=\Delta H+(H \cdot \nabla) v, & \text { in } \Omega(t), \\
			&\operatorname{div} v=0, \quad \operatorname{div} H=0, &  \text { in } \Omega(t),
		\end{aligned}\right.
	\end{align}
	which governs the dynamics of incompressible viscous conducting fluids under magnetic effects and resistivity in an electromagnetic field. Here,   $v = (v_{1},  \dots,  v_{N} )$, $H = (H_{1},  \dots,  H_{N} )$, and $\Omega(t)\in \mathbb{R}^{N}$ represent the fluid
	velocity, the magnetic field, and the evolving domain in $\mathbb{R}^{N}$  with  $N\geq2$, respectively. $\mathbf{D}(v)$ denotes the doubled deformation tensor with component $\partial_i v_j+\partial_j v_i$,  $\mathbf{I}$ is the  identity matrix, and $p$ represents fluid pressure. 
	
	Let  $\Gamma(t)$ denote  the boundary of $\Omega(t)$ and  $\mathbf{n}_{t}$ the outward unit normal to $\Gamma(t)$.
	Due to resistivity, it is natural to impose a zero magnetic field condition on the free boundary. Thus, we consider equations (\ref{1.1}) supplemented with the following initial and boundary conditions:
	\begin{equation}\label{1.2}
		\begin{cases}( \mathbf{D}(v) -P\mathbf{I}) \mathbf{n}_{t}=(H \otimes H ) \mathbf{n}_{t}, & \quad \text { on } \Gamma(t),\\
			H =0, & \quad \text { on } \Gamma(t), \\ 
			v|_{t=0}=v_0,\quad  H|_{t=0}=H_0, & \quad \text { in } \Omega_0, \end{cases}
	\end{equation}
	where  $P=p+\frac{1}{2}|H|^2$ denotes  the total pressure. We do not consider the effects of gravity or surface tension here. Since $H  = 0 $ implies  $(H \otimes H ) \mathbf{n}_{t}=0$ on $\Gamma(t)$,  the boundary conditions in (\ref{1.2}) simplify to
	\begin{equation}\label{1.3}
		( \mathbf{D}(v) -P\mathbf{I}) \mathbf{n}_{t}=0, \quad  H =0, \qquad \text{on}  \qquad  \Gamma(t).
	\end{equation}
	Additionally, in Section \ref{sec5}, we consider the scenario where the magnetic field outside the fluid is a constant vector, as previously studied in \cite{GZN}, rather than $H=0$ on $\Gamma(t)$.

In recent decades, significant strides have been made in establishing the well-posedness of the free boundary problems in fluid dynamics, particularly in the contexts of the incompressible Euler and Navier-Stokes equations, as noted in \cite{Wu2011,Germain2012,Coutand2007,Lindblad2005,Ionescu2015, JPG, JPG1,TS3} and related works.

However, the MHD model holds even greater relevance in physics due to its application in electrically conductive fluids and magnetic fields. Here, we review notable advancements concerning the free boundary problem within the MHD framework. Hao and Luo \cite{Hao2014} contributed significantly by establishing a priori estimates for the free boundary problem in incompressible ideal MHD, particularly under the Taylor-type sign condition. They further demonstrated the local well-posedness of the linearized problem \cite{Hao2021} and highlighted the problem's ill-posed nature in the two-dimensional scenario when the Taylor-type sign condition is violated \cite{Hao2020}.

Lee \cite{DL1} extended this research by deriving uniform estimates for the viscous free boundary MHD problem in infinitely deep domains under constant gravitational forces. Moreover, in domains of finite depth, he demonstrated the local well-posedness for the ideal case, considering both vanishing kinematic viscosity and magnetic diffusivity \cite{DL}. Luo and Zhang \cite{Luo2021} contributed by deriving a priori estimates for the free boundary incompressible ideal MHD equations incorporating surface tension, with local well-posedness established by Gu, Luo, and Zhang \cite{Gu2023}.

Guo, Zeng, and Ni \cite{GZN} explored decay rates in the context of viscous incompressible MHD equations, both with and without surface tension. Sun, Wang, and Zhang \cite{Sun2019} tackled the local well-posedness of the plasma-vacuum problem within the incompressible ideal MHD equations, focusing on the non-collinear stability condition. Wang and Xin \cite{Wang2021} addressed global well-posedness considerations, factoring in magnetic diffusion. Li and Li \cite{Li2022} investigated the local well-posedness of the two-phase problem in incompressible MHD, considering the presence of surface tension. Zhao \cite{Zhao2024} contributed by establishing local well-posedness for incompressible ideal MHD equations using various stability mechanisms in Eulerian coordinates.

Despite these advancements, current literature lacks research specifically on the free boundary incompressible resistive MHD equations within critical spaces and maximal $L^1$-regularity. This gap contrasts with the extensive research on the Navier-Stokes equations \cite{Chen2015,Gallagher2016,TS1,TS2} and related references, which have seen significant exploration in similar contexts.
	
We now introduce Lagrangian coordinates, where coordinates are constant along the integral curves of the velocity vector field, thus fixing the boundary in these coordinates. The Lagrangian transformation is defined as follows:
	\begin{equation}\label{1.4}
		\frac{d x}{d t}=v(t, x(t, y)),\;      x(0,y)=y, \; y\in \Omega_0.
	\end{equation}
		
	 If $ v(x)$ is Lipschitz continuous, the existence and uniqueness of solutions for the ordinary differential equation ensures that (\ref{1.4}) can be uniquely solved by:
	\begin{equation}\label{1.5}
		x(t, y)=y+\int_0^t v(\tau, x(\tau, y)) d \tau.
	\end{equation}
	
	Define
	$$u(t, y) \equiv v(t, x(t, y)), \quad q(t, y) \equiv P(t, x(t, y)),  \quad B(t, y) \equiv H(t, x(t, y)).
	$$
	Let $\frac{dx}{dy}$ denote  the Jacobian matrix of the transformation $x(t, y)=y+\int_0^t u(\tau, y) d \tau$, that is,
	$$\frac{\partial x_i}{\partial y_j}=\delta_{ij}+\int_0^t (Du)_{ij} d\tau,\quad  (Du)_{ij}=\partial_{j}u_{i}. $$
	
	Since $\delta_{j}^{ i}=\frac{\partial y_i}{\partial y_j}=\sum_{k=1}^N \frac{\partial y_i}{\partial x_k} \frac{\partial x_k}{\partial y_j}$,
	where $\partial_{i}:=\partial_{y_i}$, let $A=(\frac{dx}{dy})^{-1}$  denote  the inverse of the Jacobian matrix $\frac{dx}{dy}$, namely
	\begin{equation}\label{1.6}
		A=\left(\begin{array}{cccc}
			\frac{\partial y_1}{\partial x_1} & \frac{\partial y_1}{\partial x_2} & \cdots & \frac{\partial y_1}{\partial x_n} \\
			\frac{\partial y_2}{\partial x_1} & \frac{\partial y_2}{\partial x_2} & \cdots & \frac{\partial y_2}{\partial x_n} \\
			\vdots & \vdots & \cdots & \vdots \\
			\frac{\partial y_n}{\partial x_1} & \frac{\partial y_n}{\partial x_2} & \cdots & \frac{\partial y_n}{\partial x_n}
		\end{array}\right).
	\end{equation}
	Let $(\mathrm{cof}(\frac{dx}{dy}))_{ j}^{i}$ denote the cofactor of $\frac{dx}{dy}$ at position $(i, j)$; from (\ref{1.5}), we have $
	\operatorname{det} (\frac{dx}{dy}) \frac{\partial y_i}{\partial x_j}=(\mathrm{cof}(\frac{dx}{dy}))_{j}^{i}
	$.
	
	Since $\operatorname{ div} v=0$, it follows that $\operatorname{det} (\frac{dx}{dy})=1$. Thus,
	\begin{equation}\label{1.7}
		\begin{aligned}
			A^{i}_{j}& =\frac{\partial y_i}{\partial x_j}=(\operatorname{cof}(\frac{dx}{dy}))_{j}^{i}=(-1)^{i+j} \operatorname{det}
			\left(\begin{array}{cc|c|cc}a_{11} & \cdots & a_{1 i} & \cdots & a_{1 N} \\ \vdots & & \vdots & & \vdots \\ \hline a_{j 1} & \cdots & a_{j i} & \cdots & a_{j N} \\ \hline \vdots & & \vdots & & \vdots \\ a_{N 1} & \cdots & a_{N i} & \cdots & a_{N N}\end{array}\right)
			\\ &=\delta_{ j}^{i}+P_{j}^{i}\left(\int_0^t \nabla_i^{\prime} u_1 d \tau, \ldots, \int_0^t \nabla_i^{\prime} u_{j-1} d \tau, \int_0^t \nabla_i^{\prime} u_{j+1} d \tau, \ldots, \int_0^t \nabla_i^{\prime} u_n d \tau\right),
		\end{aligned}
	\end{equation}
	where $a_{j k}=\frac{\partial x_j}{\partial y_k}$, and  $P_{j}^{i}$ is a polynomial satisfying  $P_{ j}^{i}(0, \ldots, 0)=0$, and
	$$
	\nabla_i^{\prime} u=\left(\frac{\partial u}{  \partial y_1}, \ldots, \frac{\partial u}{ \partial y_{i-1}}, \frac{\partial u}{  \partial y_{i+1}}, \ldots, \frac{\partial u }{ \partial y_N}\right).
	$$
	
 From (\ref{1.7}), we deduce that  $P_{j}^{i}$ is a polynomial of  degree $(N-1)$ without constant terms:
	\begin{equation}\label{1.8}\begin{aligned}
			&P_{j}^{i}\left(\int_0^t \nabla_i^{\prime} u_1 d \tau, \ldots, \int_0^t \nabla_i^{\prime} u_{j-1} d \tau, \int_0^t \nabla_i^{\prime} u_{j+1} d \tau, \ldots, \int_0^t \nabla_i^{\prime} u_N d \tau\right)
			\\&=\sum_{r=1}^{N-1} c_r \prod_{m, \ell \leq N, (m,\ell) \neq(j,i)}^r\left(\int_0^t \partial_{\ell} u_m d \tau\right),
		\end{aligned}
	\end{equation}
	where $c_{r}=1$ or $-1$.
	
	Define the half-space and its boundary:
	\begin{align*}
		\mathbb{R}_{+}^N & \equiv\left\{\left(y^{\prime}, y_n\right):  y^{\prime} \in \mathbb{R}^{N-1}, y_n>0\right\}, \\
		\partial \mathbb{R}_{+}^N & \equiv \mathbb{R}^{N-1} \times\{0\}=\left\{\left(y^{\prime}, y_n\right):  y^{\prime} \in \mathbb{R}^{N-1}, y_n=0\right\}.
	\end{align*}
	We also denote  $\mathbb{R}_{-}^N$ as the negative part of $\mathbb{R}^N$, hence  $\mathbb{R}^N=\mathbb{R}_{+}^N \cup \partial \mathbb{R}_{+}^N \cup \mathbb{R}_{-}^N$.
	Consider the kinematic condition ensuring the free surface remains satisfied, where $\Gamma(t)$ represents the set of points $x=x(t, y)$ for $y \in \Gamma(0)=\partial \mathbb{R}_{+}^N$.

	 Then, equations (\ref{1.1}) and (\ref{1.3}) transform into the following equations
	\begin{equation}\label{1.9}
		\left\{\begin{aligned}
			&\partial_t u-\Delta u+\nabla q=V_{1}(u)+V_{2}(u, q)+V_{3}(u,B),  & &\quad \text{in} \quad (0,+\infty)\times \mathbb{R}_{+}^N ,
			\\ & \partial_t B-\Delta B=V_{4}(u,B), & & \quad \text{in}\quad  (0,+\infty)\times \mathbb{R}_{+}^N ,
			\\
			& \operatorname{div} u=V_{5}(u), \quad \operatorname{div} B=V_{6}(u,B), & & \quad \text{in} \quad (0,+\infty)\times \mathbb{R}_{+}^N ,
			\\
			&\left(\mathbf{D}(u)-q I\right) \mathbf{n}=V_{7}(u)+V_{8}(u, q),& & \quad \text{on}\quad (0,+\infty)\times \partial \mathbb{R}_{+}^N , \\
			&B=0, & & \quad \text{on}\quad  (0,+\infty)\times  \partial \mathbb{R}_{+}^N ,\\
			&u(0, y)=u_0(y), \quad B(0, y)=B_0(y),  & & \quad \text{in}\quad \mathbb{R}_{+}^N,
		\end{aligned}\right.
	\end{equation}
	where $\mathbf{n} = (0,\dots, 0,-1)$  denotes the outward normal. The nonlinear terms of (\ref{1.9}) are expressed as
		\begin{align}\label{1.10}
			& V_1(u)=\operatorname{div}\left(A A^{\top} \nabla u-\nabla u\right), \quad V_2(u, q)=-(A-I)^{\top} \nabla q, \nonumber \\
			& V_3(u, B)=B^i A_{i}^{j} \partial_j B, \quad V_4(u, B)=\operatorname{div}\left(A A^{\top} \nabla B-\nabla B\right)+ B^{i} A_{i}^{k}\partial_k u, \nonumber\\
			& V_5(u)=\left(A_{ i}^{j}-\delta_{i}^{j}\right) \partial_j u^{i}, \quad V_6(u, B)=\left(A_{ i}^{k}-\delta_{i}^{k}\right) \partial_k B^{i},\nonumber \\
			& V_7(u)=-A^{\top} \nabla u+(\nabla u)^{\top} A(A-I)^{\top} \mathbf{n}-\left((A-I)^{\top} \nabla u+(\nabla u)^{\top}(A-I)\right) \mathbf{n}, \nonumber\\
			& V_8(u, q)=q I(A-I)^{\top} \mathbf{n},
		\end{align}
	with $A^{\top}$ denoting the transpose of a matrix $A$,  and  Einstein’s summation convention applied. By (\ref{1.7}) and  the fact in \cite{YS} that $$\sum_{i=1}^N \frac{\partial}{\partial y_i}\left(\frac{\partial y_i}{\partial x_j}\right)=\sum_{i=1}^{N}\frac{\partial}{\partial y_i}P^{i}_{j}=0,$$
	we have
		\begin{align}\label{1.11}
			\operatorname{div}\left(A A^{\top} \nabla u-\nabla u\right) & =\sum_{i=1}^N\sum_{k=1}^N \partial_i\left(A_k^i A_k^j \partial_j u-\partial_i u\right) \nonumber \\
			& =\sum_{i=1}^N\sum_{k=1}^N \partial_i\left(\left(\delta_k^i+P_k^i\right)\left(\delta_k^j+P_k^j\right) \partial_j u-\partial_i u\right)\nonumber \\
			& =\sum_{i=1}^N\sum_{k=1}^N \partial_i\left(P_k^i P_k^j \partial_j u
			+\delta_{k}^{i}P_{k}^{j}\partial_{j}u+P_{k}^{i}\delta_{k}^{j}\partial_{j}u\right)\nonumber
			\\&=\sum_{k=1}^N P_k^i \partial_i\left( P_k^j \partial_j u\right)+\sum_{i=1}^N \partial_{i}(P_{i}^{j}\partial_{j}u)+\sum_{j=1}^NP_{j}^{i}\partial_{i}\partial_{j}u,
		\end{align}
	thus, we can reduce (\ref{1.10}) to
		\begin{align}\label{1.12}
			& V_1(u)=\sum_{k=1}^N P_k^i \partial_i\left( P_k^j \partial_j u\right)+\sum_{j=1}^NP_{j}^{i}\partial_{i}\partial_{j}u
			+\sum_{i=1}^N \partial_{i}(P_{i}^{j}\partial_{j}u), \quad V_{2i}(u, q)=-P^{j}_{i} \partial_{j}q, \nonumber\\
			&  V_3(u, B)=B^{i}\partial_{i}B+B^{i}P_{i}^{j}\partial_{j}B,\nonumber \\&
			V_4(u, B)=\sum_{k=1}^N P_k^i \partial_i\left( P_k^j \partial_j B\right)+\sum_{j=1}^NP_{j}^{i}\partial_{i}\partial_{j}B
			+\sum_{i=1}^N \partial_{i}(P_{i}^{j}\partial_{j}B)+ B^{i}\partial_{i}u+B^{i}P_{i}^{j}\partial_{j}u, \nonumber\\
			& V_5(u)=P^{j}_{i} \partial_j u^{i}, \quad V_6(u, B)= P^{k}_{i} \partial_k B^{i}, \nonumber\\
			& V_7(u)=-A^{\top} \nabla u+(\nabla u)^{\top} A(A-I)^{\top} \mathbf{n}-\left((A-I)^{\top} \nabla u+(\nabla u)^{\top}(A-I)\right) \mathbf{n},\nonumber \\
			& V_8(u, q)=q I(A-I)^{\top} \mathbf{n}.
		\end{align}

	We denote $\mathbb{R}_{+}=(0, \infty)$ as the half real line and $\overline{\mathbb{R}_{+}}=[0, \infty)$ as its closure. Let $C(I;  X)$ denote the set of all bounded continuous functions from an interval $I$ to a Banach space $X$.
	
	Now, the scaling invariance for (\ref{1.1}) reads, for all $\lambda > 0$,
	$$
	v \rightarrow  \lambda v\left(\lambda^2 t, \lambda x\right), \quad
	P \rightarrow \lambda^2 P\left(\lambda^2 t, \lambda x\right), \quad
	H \rightarrow  \lambda H\left(\lambda^2 t, \lambda x\right),
	$$
	which roughly implies that the critical spaces for the velocity, total pressure, and magnetic field remain the same as in the homogeneous case. It is well known from \cite{HT} that the incompressible Navier-Stokes equations can be globally solved in time in the invariant Bochner-Sobolev space $L^\rho\left(\mathbb{R}_{+};  \dot{H}_p^s\left(\mathbb{R}^N;  \mathbb{R}^N\right)\right)$ where $	\frac{2}{\rho}+\frac{N}{p}=1+s$.
	
	The main results are stated as follows, with details of the functional spaces used provided in the appendix.
	
	\begin{theorem}\label{THM1}
		Let $p\in [N, 2N-1)$. Suppose the initial data $u_0, B_{0} \in \dot{B}_{p, 1}^{-1+N / p}(\mathbb{R}_{+}^N)$  satisfy
		$$
		\left\|u_0\right\|_{\dot{B}_{p, 1}^{-1+\frac{N}{p}}(\mathbb{R}_{+}^N)}+\left\|B_0\right\|_{\dot{B}_{p, 1}^{-1+\frac{N}{p}}(\mathbb{R}_{+}^N)} \leq \epsilon_{1},
		$$
		for some small $\epsilon_{1} > 0$, and $\operatorname{div} u_{ 0}=\operatorname{div} B_{ 0}
		= 0 $ in the sense of distribution. Then equation (\ref{1.9}) admits a unique global solution $(u,q,B)$ satisfying the following properties:
	\begin{align*}
			& u,B \in C\left(\overline{\mathbb{R}_{+}};  \dot{B}_{p, 1}^{-1+\frac{N}{p}}(\mathbb{R}_{+}^N)\right) \cap \dot{W}^{1,1}\left(\mathbb{R}_{+};  \dot{B}_{p, 1}^{-1+\frac{N}{p}}(\mathbb{R}_{+}^N)\right), \\
			& D^{2} u, D^{2} B, \nabla q \in L^1\left(\mathbb{R}_{+};  \dot{B}_{p, 1}^{-1+\frac{N}{p}}(\mathbb{R}_{+}^N)\right), \\
			& \left.q\right|_{y_N=0} \in \dot{F}_{1,1}^{\frac{1}{2}-\frac{1}{2 p}}\left(\mathbb{R}_{+};  \dot{B}_{p, 1}^{-1+\frac{N}{p}}\left(\mathbb{R}^{N-1}\right)\right) \cap L^1\left(\mathbb{R}_{+};  \dot{B}_{p, 1}^{\frac{N-1}{p}}\left(\mathbb{R}^{N-1}\right)\right),
		\end{align*}
		with the estimates
			\begin{align}\label{1.13}
				& \left\|\partial_t u\right\|_{L^1(\mathbb{R}_{+};  \dot{B}_{p, 1}^{-1+\frac{N}{p}}(\mathbb{R}_{+}^N))}+\left\|\partial_t B\right\|_{L^1(\mathbb{R}_{+};  \dot{B}_{p, 1}^{-1+\frac{N}{p}}(\mathbb{R}_{+}^N))}+\left\|D^2 u\right\|_{L^1(\mathbb{R}_{+};  \dot{B}_{p, 1}^{-1+\frac{N}{p}}(\mathbb{R}_{+}^N))}\nonumber\\
				& +\left\|D^2 B\right\|_{L^1(\mathbb{R}_{+};  \dot{B}_{p, 1}^{-1+\frac{N}{p}}(\mathbb{R}_{+}^N))}+\|\nabla q\|_{L^1(\mathbb{R}_{+};  \dot{B}_{p, 1}^{-1+\frac{N}{p}}(\mathbb{R}_{+}^N))}\nonumber \\&+\left\|\left.q\right|_{y_N=0}\right\|_{\dot{F}_{1,1}^{\frac{1}{2}-\frac{1}{2 p}}(\mathbb{R}_{+};  \dot{B}_{p, 1}^{-1+\frac{N}{p}}(\mathbb{R}^{N-1}))}+\left\|\left.q\right|_{y_N=0}\right\|_{L^1(\mathbb{R}_{+};  \dot{B}_{p, 1}^{\frac{N-1}{p}}(\mathbb{R}^{N-1}))} \leq \epsilon,
			\end{align}
		where $D^2 $ denotes all the second-order derivatives with respect to $y$, and $\epsilon$ is a constant depending only on $N, p$ and $\epsilon_{1}$.
	\end{theorem}

The structure of this paper is outlined as follows.	Our proof of Theorem \ref{THM1} strongly relies on the end-point estimate of maximal regularity for the initial-boundary value problem of the resistive MHD in the half-space $\mathbb{R}_{+}^{N}$. Thus, in Section \ref{sec2}, we recall the fact that the maximal $L^{1}$-regularity theorem for the Stokes equations which have been obtained in \cite{TS} and prove $L^{1}$-regularity theorem for heat equations with zero boundary conditions. Then, Section \ref{sec3} focuses on estimating the nonlinear terms appearing in the resistive MHD equations.
In Section \ref{sec4}, we derive estimates for the difference of nonlinear terms and utilize the contraction mapping principle to establish the global well-posedness of the resistive MHD equations.
Section \ref{sec5} extends the results of Section \ref{sec4} to consider the case where a constant magnetic field is present outside the fluid, using similar analytical techniques.
Finally, in the Appendix, we gather and define important notations used throughout the paper, and summarize key results for the reader's convenience.

	\section{ Linearization Theory}\label{sec2}
	We investigate  maximal $L^{1}$-regularity to two linearized equations: the  Stokes equations with transmission conditions on $\partial\mathbb{R}_{+}^{N}$,  and the
	other is the heat equations subject to zero magnetic boundary conditions.
	
	\subsection{Maximal regularity for the generalized Stokes equations in the half-space}
	
	This subsection is devoted to presenting maximal $L^{1}$-regularity for the Stokes equations corresponding (\ref{1.9}) with inhomogeneous free  boundary condition:
	\begin{equation}\label{2.1}
		\left\{\begin{aligned}
			&\partial_t u-\Delta u+\nabla q=F_{1},  & &\quad \text{in} \quad \mathbb{R}_{+}\times \mathbb{R}_{+}^N ,
			\\
			& \operatorname{div} u=F_{2},  & & \quad \text{in} \quad \mathbb{R}_{+}\times \mathbb{R}_{+}^N ,
			\\
			&\left(\mathbf{D}(u)-q I\right) \mathbf{n}=F_{3},& & \quad \text{on}\quad \mathbb{R}_{+}\times \partial \mathbb{R}_{+}^N, \\
			&u(0, y)=u_0(y),  & & \quad \text{in}\quad \mathbb{R}_{+}^N,
		\end{aligned}\right.
	\end{equation}
	where $\mathbf{n} = (0,\dots, 0,-1)$, $F_{1}$, $F_{2}$ and $F_{3}$ are given functions.  We can obtain the following maximal $L^1$-regularity result.
	
	\begin{theorem}[cf. \cite{TS}]\label{thm2.1}  Let $1<p<\infty$ and $-1+1 / p<s \leq 0$.  Let $u_0 \in \dot{B}_{p, 1}^s(\mathbb{R}_{+}^N)$  be initial data for equations (\ref{2.1}) with $\operatorname{div} u_0=\left.F_{2}\right|_{t=0}$  in  the sense of distribution and let $F_{1}$, $F_{2}$ and $F_{3}$ be functions appearing on the right side of equations (\ref{2.1}) and satisfying the conditions:
		\begin{align*}
			&F_{1} \in L^1\left(\mathbb{R}_{+};  \dot{B}_{p, 1}^s(\mathbb{R}_{+}^N)\right), \quad
			\nabla F_{2} \in L^1\left(\mathbb{R}_{+};  \dot{B}_{p, 1}^s(\mathbb{R}_{+}^N)\right), \quad\nabla(-\Delta)^{-1} F_{2} \in \dot{W}^{1,1}\left(\mathbb{R}_{+};  \dot{B}_{p, 1}^s(\mathbb{R}_{+}^N)\right), \\
			& F_{3} \in \dot{F}_{1,1}^{\frac{1}{2}-\frac{1}{2 p}}\left(\mathbb{R}_{+};  \dot{B}_{p, 1}^s\left(\mathbb{R}^{N-1}\right)\right) \cap L^1\left(\mathbb{R}_{+};  \dot{B}_{p, 1}^{s+1-\frac{1}{p}}\left(\mathbb{R}^{N-1}\right)\right).
		\end{align*}
		
		Then, problem (\ref{2.1}) admits a unique solution $(u, q)$ with
		\begin{align*}
			& u \in C\left(\overline{\mathbb{R}_{+}};  \dot{B}_{p, 1}^s(\mathbb{R}_{+}^N)\right) \cap \dot{W}^{1,1}\left(\mathbb{R}_{+};  \dot{B}_{p, 1}^s(\mathbb{R}_{+}^N)\right), \quad D^{2} u, \nabla q \in L^1\left(\mathbb{R}_{+};  \dot{B}_{p, 1}^s(\mathbb{R}_{+}^N)\right), \\
			& \left.q\right|_{y_n=0} \in \dot{F}_{1,1}^{\frac{1}{2}-\frac{1}{2 p}}\left(\mathbb{R}_{+};  \dot{B}_{p, 1}^s\left(\mathbb{R}^{N-1}\right)\right) \cap L^1\left(\mathbb{R}_{+};  \dot{B}_{p, 1}^{s+1-\frac{1}{p}}\left(\mathbb{R}^{N-1}\right)\right),
		\end{align*}
		possessing the estimates
			\begin{align*}
				& \left\|\partial_t u\right\|_{L^1\left(\mathbb{R}_{+};  \dot{B}_{p, 1}^s(\mathbb{R}_{+}^N)\right)}+\left\|D^2 u\right\|_{L^1(\mathbb{R}_{+};  \dot{B}_{p, 1}^s(\mathbb{R}_{+}^N))}+\|\nabla p\|_{L^1(\mathbb{R}_{+};  \dot{B}_{p, 1}^s(\mathbb{R}_{+}^N))}\\
				& +\left\|\left.q\right|_{y_n=0}\right\|_{\dot{F}_{1,1}^{\frac{1}{2}-\frac{1}{2 p}}(\mathbb{R}_{+};  \dot{B}_{p, 1}^s\left(\mathbb{R}^{N-1}\right))}+\left\|\left.q\right|_{y_n=0}\right\|_{L^1(\mathbb{R}_{+};  \dot{B}_{p, 1}^{s+1-\frac{1}{p}}\left(\mathbb{R}^{N-1}\right))} \\
				& \leq C\bigg(\left\|u_0\right\|_{\dot{B}_{p, 1}^s(\mathbb{R}_{+}^N)}+\|F_{1}\|_{L^1(\mathbb{R}_{+};  \dot{B}_{p, 1}^s(\mathbb{R}_{+}^N))}
				+\|\nabla F_{2}\|_{L^1\left(\mathbb{R}_{+};  \dot{B}_{p, 1}^s(\mathbb{R}_{+}^N)\right)}\\
				&\qquad\qquad+\left\|\partial_t \nabla(-\Delta)^{-1} F_{2}\right\|_{L^1(\mathbb{R}_{+};  \dot{B}_{p, 1}^s(\mathbb{R}_{+}^N))}+\|F_{3}\|_{\dot{F}_{1,1}^{\frac{1}{2}-\frac{1}{2 p}}\left(\mathbb{R}_{+};  \dot{B}_{p, 1}^s\left(\mathbb{R}^{N-1}\right)\right)}\\&\qquad\qquad +\|F_{3}\|_{L^1(\mathbb{R}_{+};  \dot{B}_{p, 1}^{s+1-\frac{1}{p}}\left(\mathbb{R}^{N-1}\right))}\bigg),
			\end{align*}
		where $C>0$  is a  constant  depending only on $p, s$ and $N$.
	\end{theorem}

	\subsection{Maximal regularity for the heat equations with zero magnetic boundary condition in the half-space}
	
	This subsection is devoted to presenting the maximal $L^{1}$ regularity for the heat
	equations with zero boundary conditions:
	\begin{equation}\label{2.3}
		\left\{\begin{aligned}
			& \partial_t B-\Delta B=f,  & &\quad \text{in} \quad \mathbb{R}_{+}\times \mathbb{R}_{+}^N ,
			\\
			&\operatorname{div} B=g, & &\quad \text{in} \quad \mathbb{R}_{+}\times \mathbb{R}_{+}^N ,
			\\
			& B=0, & & \quad \text{on}\quad \mathbb{R}_{+}\times \partial \mathbb{R}_{+}^N,\\
			& B(0, y)=B_0(y), & &\quad \text{in} \quad \mathbb{R}_{+}^N.
		\end{aligned}\right.
	\end{equation}
	
	Firstly, we shall remove the divergence data. Introducing the even extension of divergence data $g$ with respect to $y_N$:
	\begin{equation}\label{2.4}
		\bar{g}(t, y)=\left\{
		\begin{aligned}
			&g\left(t, y^{\prime}, y_n\right), &   y_N>0, \\
			&g\left(t, y^{\prime},-y_n\right), &   y_N<0,
		\end{aligned}\right.
	\end{equation}
	with $y^{\prime}=\left(y_1, y_2, \cdots, y_{N-1}\right)$, we consider the following problem
	\begin{equation}\label{2.5}
		\left\{
		\begin{aligned}
			&-\Delta \rho=\bar{g}, & t>0, \quad  &  y \in \mathbb{R}^N, \\
			&\rho|_{y_N=0}=0, &  t>0, \quad &  y^{\prime} \in \mathbb{R}^{N-1}.
		\end{aligned}\right.
	\end{equation}
	Here, we assume $\underset{{t \rightarrow \infty}}\lim  \rho\left(t, y^{\prime}, y_{N}\right)=0$ for almost everywhere $\left(y^{\prime}, y_{N}\right) \in \mathbb{R}^{N-1} \times \mathbb{R}_{+}$, one solution of (\ref{2.5}) is given by the Newtonian potential $\rho=(-\Delta)^{-1} \bar{g} \equiv \Gamma * \bar{g}$ with the Newtonian kernel $\Gamma$ in $\mathbb{R}^N$
	$$
	\Gamma(y)=\left\{
	\begin{aligned}
		&\frac{1}{2 \pi} \log |y|^{-1}, & N=2, \\
		&\left((N-2) \omega_N\right)^{-1}|y|^{2-N}, & N \geq 3,
	\end{aligned}
	\right.
	$$
	where $\omega_N$ is the surface area of the unit sphere in $\mathbb{R}^N$.
	Then for $1<p<\infty$ and $-1+1 / p<s<1 / p$,  $\nabla \rho$ satisfies the estimate
	\begin{equation}\label{2.6}
		\left\{\begin{aligned}
			&\left\|\nabla^3 \rho\right\|_{L^1(\mathbb{R}_{+};  \dot{B}_{p, 1}^s(\mathbb{R}_{+}^N))} \leq C\|\nabla g\|_{L^1(\mathbb{R}_{+};  \dot{B}_{p, 1}^s(\mathbb{R}_{+}^N))}, \\
			&\left\|\partial_t \nabla \rho\right\|_{L^1(\mathbb{R}_{+};  \dot{B}_{p, 1}^s(\mathbb{R}_{+}^N))} \leq C\left\|\partial_t(-\Delta)^{-1} \nabla g\right\|_{L^1(\mathbb{R}_{+};  \dot{B}_{p, 1}^s(\mathbb{R}_{+}^N))}.
		\end{aligned}\right.
	\end{equation}
	
	Indeed, the corresponding estimate to (\ref{2.6}) in $\mathbb{R}^{N}$ follows directly from the elliptic estimate for the Poisson equation  and hence  estimate (\ref{2.6}) in the half-space naturally follows from the definition of the Besov space in $\mathbb{R}^{N}_{+}$.
	
	\begin{remark}
		For $1<p<\infty$ and $-1+1 / p<s<1 / p$, let 
		$$B \in C\left(\overline{\mathbb{R}_{+}};  \dot{B}_{p, 1}^s(\mathbb{R}_{+}^N)\right) \cap\dot{W}^{1,1}\left(\mathbb{R}_{+};  \dot{B}_{p, 1}^s(\mathbb{R}_{+}^N)\right),\; D^{2} B \in L^1\left(\mathbb{R}_{+};  \dot{B}_{p, 1}^s(\mathbb{R}_{+}^N)\right).$$ 
		From Theorem \ref{thmA1},  we have
		$$
		\lim _{t \rightarrow \infty} B\left(t, y^{\prime}, y_{N}\right)=0, \quad \text { a.e. }\left(y^{\prime}, y_{N}\right) \in \mathbb{R}^{N-1} \times \mathbb{R}_{+}.
		$$
		Thus, it is reasonable to assume $\underset{t \rightarrow \infty}\lim \rho\left(t, y^{\prime}, y_{N}\right)=0$ for a.e. $\left(y^{\prime}, y_{N}\right) \in \mathbb{R}^{N-1} \times \mathbb{R}_{+}$.
	\end{remark}
	Then, let $w=B+\left.\nabla \rho\right|_{\mathbb{R}_{+}^N}$, the  function $w$ satisfies the equations
	\begin{equation}\label{2.7}
		\left\{\begin{aligned}
			& \partial_t w-\Delta w=f+\left.\left(\partial_t \nabla \rho-\Delta \nabla \rho\right)\right|_{y_N>0}, & &\quad \text{in} \quad \mathbb{R}_{+}\times \mathbb{R}_{+}^N ,\\
			&\operatorname{div} w=0, & &\quad \text{in} \quad \mathbb{R}_{+}\times \mathbb{R}_{+}^N , \\
			& w=B+\nabla \rho, & &\quad \text{on}\quad \mathbb{R}_{+}\times \partial \mathbb{R}_{+}^N, \\
			&w(0, y)=B_0(y)+\left.\nabla \rho(0, y)\right|_{y_N>0}, & &\quad \text{in} \quad \mathbb{R}_{+}^N.
		\end{aligned}\right.
	\end{equation}
In order to extend them into $\mathbb{R}^{N}$, we consider the following  Cauchy problem via the even extension:
	\begin{equation}\label{2.8}
		\left\{\begin{aligned}
			& \partial_t \bar{B}-\Delta \bar{B}=\bar{f}+\partial_t \nabla\rho-\Delta \nabla\rho, & &\quad \text{in} \quad \mathbb{R}_{+}\times \mathbb{R}^N , \\
			&\operatorname{div} \bar{B}=0, & &\quad \text{in} \quad \mathbb{R}_{+}\times \mathbb{R}^N , \\
			& \left.\bar{B}\right|_{t=0}=\bar{B}_{0}(y)+\nabla \rho(0,y), & &\quad \text{in} \quad \mathbb{R}^N.
		\end{aligned}\right.
	\end{equation}
	For the initial data $B_{0}(y)$,  we also employ the same extension to get $\bar{B}_{0}(y)$. Then it is known that the solution $\bar{B}$ of equation (\ref{2.8}) satisfies maximal $L^1$-regularity. Indeed, for any $-1+1 / p<s<1 / p$ and $1<p<\infty$,  by \cite{RP}, we have
		\begin{align}\label{2.9}
			&\left\|\partial_t \bar{B}\right\|_{L^1(\mathbb{R}_{+};  \dot{B}_{p, 1}^s\left(\mathbb{R}^N\right))}+\left\|D^2 \bar{B}\right\|_{L^1(\mathbb{R}_{+};  \dot{B}_{p, 1}^s\left(\mathbb{R}^N\right))}\nonumber \\
			\leq& C\left(\left\|\bar{B}_0\right\|_{\dot{B}_{p, 1}^s\left(\mathbb{R}^N\right)}+\|\bar{f}\|_{L^1(\mathbb{R}_{+};  \dot{B}_{p, 1}^s\left(\mathbb{R}^N\right))}+\|\partial_t \nabla\rho-\Delta \nabla\rho\|_{L^1(\mathbb{R}_{+};  \dot{B}_{p, 1}^s\left(\mathbb{R}^N\right))}\right).
		\end{align}
	Restricting the solution $\bar{B}$ to the half-space $\mathbb{R}_{+}^N$, we directly obtain from (\ref{2.6}) and (\ref{2.9}) that
		\begin{align}\label{2.10}
			&\left\|\partial_t \bar{B}\right\|_{L^1(\mathbb{R}_{+};  \dot{B}_{p, 1}^s(\mathbb{R}_{+}^N))}+  \left\|D^2 \bar{B}\right\|_{L^1(\mathbb{R}_{+};  \dot{B}_{p, 1}^s(\mathbb{R}_{+}^N))} \nonumber\\
			\leq & C\left(\left\|B_0\right\|_{\dot{B}_{p, 1}^s(\mathbb{R}_{+}^N)}+\|f\|_{L^1(\mathbb{R}_{+};  \dot{B}_{p, 1}^s(\mathbb{R}_{+}^N))}\right. \nonumber\\
			& \left.\qquad+\|\nabla g\|_{L^1(\mathbb{R}_{+};  \dot{B}_{p, 1}^s(\mathbb{R}_{+}^N))}+\left\|\partial_t \nabla(-\Delta)^{-1} g\right\|_{L^1(\mathbb{R}_{+};  \dot{B}_{p, 1}^s(\mathbb{R}_{+}^N))}\right).
		\end{align}

	Finally, let $ \mathbf{B}=w-\left.\bar{B}\right|_{y_n>0} \equiv B+\left.\nabla \rho\right|_{y_n>0}-\left.\bar{B}\right|_{y_n>0}$. We reduce the original problem into the following initial-boundary value problem for $\mathbf{B}$ :
	\begin{equation}\label{2.11}
		\left\{\begin{aligned}
			&\partial_t \mathbf{B}-\Delta \mathbf{B}=0, & &\quad \text{in} \quad \mathbb{R}_{+}\times \mathbb{R}_{+}^N , \\
			& \operatorname{div}  \mathbf{B}=0, & &\quad \text{in} \quad \mathbb{R}_{+}\times \mathbb{R}_{+}^N , \\
			& \mathbf{B}=\nabla \rho+\bar{B},  & &\quad \text{on} \quad \mathbb{R}_{+}\times \partial\mathbb{R}_{+}^N , \\
			&\mathbf{B}(0, y)=0, & &\quad \text{in} \quad \mathbb{R}_{+}^N.
		\end{aligned}\right.
	\end{equation}
	Then,  we consider the following equations to get the maximal $L^1$-regularity for (\ref{2.11}).
	\begin{equation}\label{2.12}
		\left\{\begin{aligned}
			&\partial_t b-\Delta b=f, & &\quad \text{in} \quad \mathbb{R}_{+}\times \mathbb{R}_{+}^N , \\
			& b\left(t, y^{\prime}, y_n\right)|_{y_n=0}=h\left(t, y^{\prime}\right),& &\quad \text{on} \quad \mathbb{R}_{+}\times \partial\mathbb{R}_{+}^N , \\
			&b(0, y)=b_{0}, & &\quad \text{in} \quad \mathbb{R}_{+}^N.
		\end{aligned}\right.
	\end{equation}

	Indeed, we have the following  the maximal theorem  for the Stokes equations with boundary conditions.
	\begin{theorem}[cf. \cite{TS1}]\label{thm2.3}
		Let $1 < p <\infty$ and $-1 + 1/p < s \leq 0$. (\ref{2.12}) admits a unique solution
		$$
		b \in \dot{W}^{1,1}\left(\mathbb{R}_{+};  \dot{B}_{p, 1}^s(\mathbb{R}_{+}^N)\right), \quad D^{2} b \in L^1\left(\mathbb{R}_{+};  \dot{B}_{p, 1}^s(\mathbb{R}_{+}^N)\right)
		$$
		if and only if the external, initial and boundary data in (\ref{2.12}) satisfy
		\begin{align*}
			& f \in L^1\left(\mathbb{R}_{+};  \dot{B}_{p, 1}^s(\mathbb{R}_{+}^N)\right), \quad b_0 \in \dot{B}_{p, 1}^s(\mathbb{R}_{+}^N), \\
			&  h \in \dot{F}_{1,1}^{1-1 / 2 p}\left(\mathbb{R}_{+};  \dot{B}_{p, 1}^s\left(\mathbb{R}^{N-1}\right)\right) \cap L^1\left(\mathbb{R}_{+};  \dot{B}_{p, 1}^{s+2-1 / p}\left(\mathbb{R}^{N-1}\right)\right),
		\end{align*}
		respectively. Then, the solution $b$ satisfies the following estimate for some constant $C>0$ depending only on $p, s$ and $N$
			\begin{align*}
				&\left\|\partial_t b\right\|_{L^1\left(\mathbb{R}_{+}; \dot{B}_{p,1}^s(\mathbb{R}_{+}^N)\right)}+\left\|D^2 b \right\|_{L^1(\mathbb{R}_{+};  \dot{B}_{p, 1}^s(\mathbb{R}_{+}^N))}\\ 
				\leq & C\left(\left\|b_0\right\|_{\dot{B}_{p, 1}^s(\mathbb{R}_{+}^N)}+\|f\|_{L^1(\mathbb{R}_{+};  \dot{B}_{p, 1}^s(\mathbb{R}_{+}^N)}+\|h\|_{\dot{F}_{1,1}^{1-1 / 2 p}(\mathbb{R}_{+};  \dot{B}_{p, 1}^s(\mathbb{R}^{N-1}))}\right.\left.+\|h\|_{L^1(\mathbb{R}_{+};  \dot{B}_{p, 1}^{s+2-1 / p}(\mathbb{R}^{N-1}))}\right).
			\end{align*}
	\end{theorem}
	
	By Theorems \ref{thmA1} and \ref{thm2.3}, (\ref{2.6}) and (\ref{2.10}), we  can get the maximal $L^{1}$-regularity for (\ref{2.11}). For $1 < p < \infty $ and $-1 + 1/p < s \leq 0$, we have
		\begin{align}\label{2.14}
			&\left\|\partial_t \mathbf{B}\right\|_{L^1(\mathbb{R}_{+}; \dot{B}_{p,1}^s(\mathbb{R}_{+}^N))}+\left\| D^2 \mathbf{B} \right\|_{L^1(\mathbb{R}_{+};  \dot{B}_{p, 1}^s(\mathbb{R}_{+}^N))} \nonumber\\
			\leq & C\left(\|\nabla\rho\|_{\dot{F}_{1,1}^{1-1 / 2 p}(\mathbb{R}_{+};  \dot{B}_{p, 1}^s(\mathbb{R}^{N-1}))}+\|\nabla\rho\|_{L^1(\mathbb{R}_{+};  \dot{B}_{p, 1}^{s+2-1 / p}(\mathbb{R}^{N-1}))}\|\bar{B}\|_{\dot{F}_{1,1}^{1-1 / 2 p}(\mathbb{R}_{+};  \dot{B}_{p, 1}^s(\mathbb{R}^{N-1}))}\right.\nonumber\\
			& \left.\qquad+\|\bar{B}\|_{L^1(\mathbb{R}_{+};  \dot{B}_{p, 1}^{s+2-1 / p}(\mathbb{R}^{N-1}))}\right)\nonumber
			\\\leq & C\left(\left\|B_0\right\|_{\dot{B}_{p, 1}^s(\mathbb{R}_{+}^N)}+\|f\|_{L^1(\mathbb{R}_{+};  \dot{B}_{p, 1}^s(\mathbb{R}_{+}^N))}+\|\nabla g\|_{L^1(\mathbb{R}_{+};  \dot{B}_{p, 1}^s(\mathbb{R}_{+}^N))}\right.\nonumber\\
			&\qquad\left.+\left\|\partial_t \nabla(-\Delta)^{-1} g\right\|_{L^1(\mathbb{R}_{+};  \dot{B}_{p, 1}^s(\mathbb{R}_{+}^N))}\right).
		\end{align}
	
	We can prove the maximal $L^1$-regularity for the original equations (\ref{2.3}) by combining with those estimates (\ref{2.6}), (\ref{2.10}), (\ref{2.14}) and the relation
	\begin{equation}\label{2.15}
		B(t, y)=\bar{B}(t, y)+\mathbf{B}(t, y)-\nabla \rho(t, y), \quad \text {in} \quad \mathbb{R}_{+}\times \mathbb{R}_{+}^N .
	\end{equation}
	
	Thus, the solution $B$ of (\ref{2.3}) satisfies the following estimates: for $1 < p < \infty $ and $-1 + 1/p < s \leq 0$,
		\begin{align}\label{2.16}
			&\left\|\partial_t B\right\|_{L^1(\mathbb{R}_{+};  \dot{B}_{p, 1}^s(\mathbb{R}_{+}^N))}+\left\|D^2 B\right\|_{L^1\left(\mathbb{R}_{+};  \dot{B}_{p, 1}^s\left(\mathbb{R}_{+}^n\right)\right)} \nonumber\\
			& \leq C\bigg(\left\|B_0\right\|_{\dot{B}_{p, 1}^s(\mathbb{R}_{+}^N)}+\|f\|_{L^1(\mathbb{R}_{+};  \dot{B}_{p, 1}^s(\mathbb{R}_{+}^N))}\nonumber\\
			&\qquad\qquad+\|\nabla g\|_{L^1(\mathbb{R}_{+};  \dot{B}_{p, 1}^s(\mathbb{R}_{+}^N))}+\left\|\partial_t \nabla(-\Delta)^{-1} g\right\|_{L^1(\mathbb{R}_{+};  \dot{B}_{p, 1}^s(\mathbb{R}_{+}^N))}\bigg).
		\end{align}

	\section{Estimates of nonlinear terms}\label{sec3}
	To obtain the estimates of $u, q$ and $B$, we shall use Theorem \ref{thm2.1} and (\ref{2.16}). For this purpose, we shall estimate the nonlinear functions appearing on the right-hand side of equations (\ref{1.9}). Let $N \leq p<2N-1$. For $\partial_t u, \partial_t B, D^2 u, D^2 B$ and $\nabla q \in$ $L^1\left(\mathbb{R}_{+};  \dot{B}_{p, 1}^{-1+N / p}\left(\mathbb{R}_{+}^n\right)\right)$, let $V_{1}(u), \dots, V_{8}(u,q)$ be the terms defined in (\ref{1.12}). By using Propositions \ref{Aprop} and \ref{Aprop1}, (\ref{1.8}) and (\ref{1.12}),  we have
	\begin{align}\label{3.1}
		& \sum_{k=1}^N\left\| P_k^i \partial_i\left( P_k^j \partial_j u\right)\right\|_{L^1(\mathbb{R}_{+};  \dot{B}_{p, 1}^{-1+\frac{N}{p}}(\mathbb{R}_{+}^N))}\nonumber \\
		& \lesssim \sup_{t>0}\left\|\int_0^t D u d \tau\right\|_{ \dot{B}_{p, 1}^{\frac{N}{p}}(\mathbb{R}_{+}^N)}\bigg(\|\partial_i\left( P_k^j \partial_j u\right)\|_{L^1(\mathbb{R}_{+};  \dot{B}_{p, 1}^{-1+\frac{N}{p}}(\mathbb{R}_{+}^N))} \nonumber\\
		&\qquad +\bigg\|\sum_{r=2}^{N-2} c_r \prod_{m, \ell \leq N}^r\left(\int_0^t \partial_{\ell} u_m d \tau\right)\partial_i\left( P_k^j \partial_j u\right)\bigg\|_{L^1(\mathbb{R}_{+};  \dot{B}_{p, 1}^{-1+\frac{N}{p}}(\mathbb{R}_{+}^N))}\bigg) \nonumber \\ &
		\lesssim\dots \nonumber
		\\ &
		\lesssim \sum_{r=1}^{N-1}\left\|\int_0^t D u d \tau\right\|_{L^\infty(\mathbb{R}_{+};  \dot{B}_{p, 1}^{\frac{N}{p}}(\mathbb{R}_{+}^N))}^r\| \nabla (P_k^j \partial_j u)\|_{L^1(\mathbb{R}_{+}; \dot{B}_{p, 1}^{-1+\frac{N}{p}}(\mathbb{R}_{+}^N))} \nonumber
		\\& \lesssim  \sum_{r=1}^{N-1}\left\|D^2 u\right\|_{L^1(\mathbb{R}_{+}; \dot{B}_{p, 1}^{-1+\frac{N}{p}}(\mathbb{R}_{+}^N))}^r\| \nabla (P_k^j \partial_j u)\|_{L^1(\mathbb{R}_{+}; \dot{B}_{p, 1}^{-1+\frac{N}{p}}(\mathbb{R}_{+}^N))},
	\end{align}
	where the symbol "$\lesssim$" denotes "$\leq C$" for some constant $C$. We see, for the last term on the right-hand side of (\ref{3.1}), that
	\begin{align*}
		&\| \nabla (P_k^j \partial_j u)\|_{L^1(\mathbb{R}_{+};  \dot{B}_{p, 1}^{-1+\frac{N}{p}}(\mathbb{R}_{+}^N))}\leq \| \nabla P_k^j \partial_j u\|_{L^1(\mathbb{R}_{+};  \dot{B}_{p, 1}^{-1+\frac{N}{p}}(\mathbb{R}_{+}^N))}
		+\| P_k^j \nabla\partial_j u\|_{L^1(\mathbb{R}_{+};  \dot{B}_{p, 1}^{-1+\frac{N}{p}}(\mathbb{R}_{+}^N))}.
	\end{align*}
	
	It is then easy to see that a similar argument of (\ref{3.1}) can be applied to estimate  $P_k^j \nabla\partial_j u$, and we conclude that
		\begin{align}\label{3.2}
			 \sum_{j=1}^N\left\| P_{j}^{i}\nabla\partial_{j}u\right\|_{L^1(\mathbb{R}_{+};  \dot{B}_{p, 1}^{-1+\frac{N}{p}}(\mathbb{R}_{+}^N))}&\lesssim  \sum_{r=1}^{N-1}\left\|D^2 u\right\|_{L^1(\mathbb{R}_{+};  \dot{B}_{p, 1}^{-1+\frac{N}{p}}(\mathbb{R}_{+}^N))}^r \| D^{2} u\|_{L^1(\mathbb{R}_{+};  \dot{B}_{p, 1}^{-1+\frac{N}{p}}(\mathbb{R}_{+}^N))} \nonumber\\
			 &\lesssim \sum_{r=2}^{N}\left\|D^2 u\right\|_{L^1(\mathbb{R}_{+};  \dot{B}_{p, 1}^{-1+\frac{N}{p}}(\mathbb{R}_{+}^N))}^r.
		\end{align}
	
	By (\ref{1.8}), we see that $P_{j}^{i}$ is a polynomial of degree $(N-1)$  without a constant term. Then, we can get that $\nabla P_{j}^{i}$ is  the product of a polynomial of degree $(N-2)$ with a constant term and $\int_{0}^{t}D^{2}ud\tau$. Thus, by Propositions \ref{Aprop} and \ref{Aprop1}, (\ref{3.1}) and H\"{o}lder's inequality,  we have
		\begin{align}\label{3.3}
			&\left\|\nabla P_{j}^{i}\nabla u\right\|_{L^1(\mathbb{R}_{+};  \dot{B}_{p, 1}^{-1+\frac{N}{p}}(\mathbb{R}_{+}^N))}\nonumber\\
			& \lesssim\left\|\nabla\left(\sum_{r=1}^{N-1} c_r \prod_{m, \ell \leq N, (m,\ell) \neq(j,i)}^r\left(\int_0^t \partial_{\ell} u_m d \tau \right)\right)\nabla u\right\|_{L^1(\mathbb{R}_{+};  \dot{B}_{p, 1}^{-1+\frac{N}{p}}(\mathbb{R}_{+}^N))}\nonumber\\& \lesssim\left(1+\sum_{r=1}^{N-2}\left\|D^2 u\right\|_{L^1(\mathbb{R}_{+};  \dot{B}_{p, 1}^{-1+\frac{N}{p}}(\mathbb{R}_{+}^N))}^r\right)\left\|\int_0^t D^{2} u d\tau \nabla u   \right\|_{L^1(\mathbb{R}_{+};  \dot{B}_{p, 1}^{-1+\frac{N}{p}}(\mathbb{R}_{+}^N))}\nonumber
			\\& \lesssim\left(1+\sum_{r=1}^{N-2}\left\|D^2 u\right\|_{L^1(\mathbb{R}_{+};  \dot{B}_{p, 1}^{-1+\frac{N}{p}}(\mathbb{R}_{+}^N))}^r\right)\left(\|\nabla u\|_{L^1(\mathbb{R}_{+};  \dot{B}_{p, 1}^{\frac{N}{p}}(\mathbb{R}_{+}^N))}\left\|\int_0^t D^2 u d \tau\right\|_{L^{\infty}(\mathbb{R}_{+};  \dot{B}_{p, 1}^{-1+\frac{N}{p}}(\mathbb{R}_{+}^N))}\right)\nonumber\\&\lesssim\left(1+\sum_{r=1}^{N-2}\left\|D^2 u\right\|_{L^1(\mathbb{R}_{+};  \dot{B}_{p, 1}^{-1+\frac{N}{p}}(\mathbb{R}_{+}^N))}^r\right)\|D^{2} u\|_{L^1(\mathbb{R}_{+};  \dot{B}_{p, 1}^{-1+\frac{N}{p}}(\mathbb{R}_{+}^N))}^{2}	\nonumber	\\
			&\lesssim \sum_{r=2}^{N} \|D^{2} u\|_{L^1(\mathbb{R}_{+};  \dot{B}_{p, 1}^{-1+\frac{N}{p}}(\mathbb{R}_{+}^N))}^{r}.
		\end{align}
	
	Combining with (\ref{3.1}),  (\ref{3.2}) and (\ref{3.3}), we have that
		\begin{align}\label{3.4}
			 &\sum_{k=1}^N\left\| P_k^i \partial_i\left( P_k^j \partial_j u\right)\right\|_{L^1(\mathbb{R}_{+};  \dot{B}_{p, 1}^{-1+\frac{N}{p}}(\mathbb{R}_{+}^N))}\nonumber\\
			 &\lesssim
			\sum_{r=1}^{N-1}\left\|D^2 u\right\|_{L^1(\mathbb{R}_{+};  \dot{B}_{p, 1}^{-1+\frac{N}{p}}(\mathbb{R}_{+}^N))}^r\sum_{r=2}^{N}\left\|D^2 u\right\|_{L^1(\mathbb{R}_{+};  \dot{B}_{p, 1}^{-1+\frac{N}{p}}(\mathbb{R}_{+}^N))}^r\nonumber
			\\ &\lesssim \sum_{r=3}^{2 N-1}\left\|D^2 u\right\|_{L^1(\mathbb{R}_{+};  \dot{B}_{p, 1}^{-1+\frac{N}{p}}(\mathbb{R}_{+}^N))}^{r}.
		\end{align}
	By (\ref{3.2}) and (\ref{3.3}), we have estimated the other terms in $V_{1} (u)$ and we conclude that
		\begin{align}\label{3.5}
			\left\|V_1(u)\right\|_{L^1(\mathbb{R}_{+};  \dot{B}_{p, 1}^{-1+\frac{N}{p}}(\mathbb{R}_{+}^N))} \lesssim  \sum_{r=2}^{2 N-1}\left\|D^2 u\right\|_{L^1(\mathbb{R}_{+};  \dot{B}_{p, 1}^{-1+\frac{N}{p}}(\mathbb{R}_{+}^N))}^{r}.
		\end{align}
	
	Next, we  consider $V_{2}(u,q)$ given in (\ref{1.12}).  Employing
	the same arguments as those in proving (\ref{3.1}), we have
		\begin{align}\label{3.6}
			\left\| V_{2}(u, q)\right\|_{L^1(\mathbb{R}_{+};  \dot{B}_{p, 1}^{-1+\frac{N}{p}}(\mathbb{R}_{+}^N))} 
			& \lesssim  \sum_{r=1}^{N-1}\left\|\int_0^t D u d \tau\right\|_{L^{\infty}(\mathbb{R}_{+};  \dot{B}_{p, 1}^{\frac{N}{p}}(\mathbb{R}_{+}^N))}^r\|\nabla q\|_{L^1(\mathbb{R}_{+};  \dot{B}_{p, 1}^{-1+\frac{N}{p}}(\mathbb{R}_{+}^N))}\nonumber\\
			&\lesssim \sum_{r=1}^{N-1}\left\|D^2 u\right\|_{L^1(\mathbb{R}_{+};  \dot{B}_{p, 1}^{-1+\frac{N}{p}}(\mathbb{R}_{+}^N))}^r\|\nabla q\|_{L^1(\mathbb{R}_{+};  \dot{B}_{p, 1}^{-1+\frac{N}{p}}(\mathbb{R}_{+}^N))}.
		\end{align}

	Let $V_3(u, B)$ be the nonlinear term given in (\ref{1.12}), by (\ref{A.3}), Proposition \ref{Aprop1} and  employing
	the same argument as in proving (\ref{3.1}), we have
		\begin{align}\label{3.7}
			&\left\| V_3(u, B)\right\|_{L^1(\mathbb{R}_{+};  \dot{B}_{p, 1}^{-1+\frac{N}{p}}(\mathbb{R}_{+}^N))}\nonumber\\
			&\lesssim \left\| B^{i}P_{i}^{j}\partial_{j}B\right\|_{L^1(\mathbb{R}_{+};  \dot{B}_{p, 1}^{-1+\frac{N}{p}}(\mathbb{R}_{+}^N))}+\left\| B^{i}\partial_{i}B\right\|_{L^1(\mathbb{R}_{+};  \dot{B}_{p, 1}^{-1+\frac{N}{p}}(\mathbb{R}_{+}^N))} \nonumber
			\\ & \lesssim  \left(\sum_{r=1}^{N-1}\left\|\int_0^t D u d \tau\right\|_{L^{\infty}(\mathbb{R}_{+};  \dot{B}_{p, 1}^{\frac{N}{p}}(\mathbb{R}_{+}^N))}^r+1\right)\|B\nabla B\|_{L^1\left(\mathbb{R}_{+};  \dot{B}_{p, 1}^{-1+\frac{N}{p}}(\mathbb{R}_{+}^N)\right)}\nonumber\\
			&\lesssim \left(\sum_{r=1}^{N-1}\left\|D^2 u\right\|_{L^1(\mathbb{R}_{+};  \dot{B}_{p, 1}^{-1+\frac{N}{p}}(\mathbb{R}_{+}^N))}^r+1\right)\left(\|\nabla B\|_{L^1(\mathbb{R}_{+};  \dot{B}_{p, 1}^{\frac{N}{p}}(\mathbb{R}_{+}^N))} \sup _{t>0}\|B\|_{\dot{B}_{p, 1}^{-1+\frac{N}{p}}(\mathbb{R}_{+}^N)}\right)\nonumber
			\\&\lesssim \left(\sum_{r=1}^{N-1}\left\|D^2 u\right\|_{L^1(\mathbb{R}_{+};  \dot{B}_{p, 1}^{-1+\frac{N}{p}}(\mathbb{R}_{+}^N))}^r+1\right)\left\|D^{2} B\right\|_{L^1(\mathbb{R}_{+};  \dot{B}_{p, 1}^{-1+\frac{N}{p}}(\mathbb{R}_{+}^N))}\left\| \int_t^{\infty} \partial_\tau B d \tau \right\|_{\dot{B}_{p, 1}^{-1+\frac{N}{p}}(\mathbb{R}_{+}^N)}\nonumber
			\\&\lesssim \left(\sum_{r=1}^{N-1}\left\|D^2 u\right\|_{L^1(\mathbb{R}_{+};  \dot{B}_{p, 1}^{-1+\frac{N}{p}}(\mathbb{R}_{+}^N))}^r+1\right)\left\|D^{2} B\right\|_{L^1(\mathbb{R}_{+};  \dot{B}_{p, 1}^{-1+\frac{N}{p}}(\mathbb{R}_{+}^N))}\left\|  \partial_t B  \right\|_{L^{1}(\mathbb{R}_{+}; \dot{B}_{p, 1}^{-1+\frac{N}{p}}(\mathbb{R}_{+}^N))}.
		\end{align}
	Next, employing the same argument as in proving (\ref{3.3}),  we have
	\begin{align}
		&\sum_{i=1}^N\left\| \nabla P_{i}^{j}\nabla B\right\|_{L^1(\mathbb{R}_{+};  \dot{B}_{p, 1}^{-1+\frac{N}{p}}(\mathbb{R}_{+}^N))} \nonumber\\
		& \lesssim\left(1+\sum_{r=1}^{N-2}\left\|D^2 u\right\|_{L^1(\mathbb{R}_{+};  \dot{B}_{p, 1}^{-1+\frac{N}{p}}(\mathbb{R}_{+}^N))}^r\right)\left(\|\nabla B\|_{L^1(\mathbb{R}_{+};  \dot{B}_{p, 1}^{\frac{N}{p}}(\mathbb{R}_{+}^N))}\left\|\int_0^t D^2 u d \tau\right\|_{L^{\infty}(\mathbb{R}_{+};  \dot{B}_{p, 1}^{-1+\frac{N}{p}}(\mathbb{R}_{+}^N))}\right),
		\notag\\
		& \lesssim \sum_{r=1}^{N-1}\left\|D^2 u\right\|_{L^1(\mathbb{R}_{+};  \dot{B}_{p, 1}^{-1+\frac{N}{p}}(\mathbb{R}_{+}^N))}^r\|D^{2} B\|_{L^1(\mathbb{R}_{+};  \dot{B}_{p, 1}^{-1+\frac{N}{p}}(\mathbb{R}_{+}^N))}.\label{3.8}
	\end{align}
	We also have that
		\begin{align}\label{3.9}
			\sum_{j=1}^N\left\|P_{j}^{i}\partial_{i}\partial_{j}B\right\|_{L^1(\mathbb{R}_{+};  \dot{B}_{p, 1}^{-1+\frac{N}{p}}(\mathbb{R}_{+}^N))}& \lesssim  \sum_{r=1}^{N-1}\left\|D^2 u\right\|_{L^1(\mathbb{R}_{+};  \dot{B}_{p, 1}^{-1+\frac{N}{p}}(\mathbb{R}_{+}^N))}^r \left\| \partial_{i} \partial_j B\right\|_{L^1(\mathbb{R}_{+};  \dot{B}_{p, 1}^{-1+\frac{N}{p}}(\mathbb{R}_{+}^N))} \nonumber
			\\&\lesssim  \sum_{r=1}^{N-1}\left\|D^2 u\right\|_{L^1(\mathbb{R}_{+};  \dot{B}_{p, 1}^{-1+\frac{N}{p}}(\mathbb{R}_{+}^N))}^r \left\| D^{2} B\right\|_{L^1(\mathbb{R}_{+};  \dot{B}_{p, 1}^{-1+\frac{N}{p}}(\mathbb{R}_{+}^N))}.
		\end{align}
	Combining with (\ref{3.8}) and (\ref{3.9}), we obtain that
		\begin{align}\label{3.10}
			&\sum_{i=1}^N\left\| \partial_{i}(P_{i}^{j}\partial_{j}B)\right\|_{L^1(\mathbb{R}_{+};  \dot{B}_{p, 1}^{-1+\frac{N}{p}}(\mathbb{R}_{+}^N))}\nonumber\\
			&\lesssim \left\| \nabla P_{i}^{j}\nabla B\right\|_{L^1(\mathbb{R}_{+};  \dot{B}_{p, 1}^{-1+\frac{N}{p}}(\mathbb{R}_{+}^N))}+\left\|P_{j}^{i}\partial_{i}\partial_{j}B\right\|_{L^1(\mathbb{R}_{+};  \dot{B}_{p, 1}^{-1+\frac{N}{p}}(\mathbb{R}_{+}^N))}\nonumber
			\\&\lesssim  \sum_{r=1}^{N-1}\left\|D^2 u\right\|_{L^1\mathbb{R}_{+};  \dot{B}_{p, 1}^{-1+\frac{N}{p}}(\mathbb{R}_{+}^N))}^r \| D^{2} B\|_{L^1(\mathbb{R}_{+};  \dot{B}_{p, 1}^{-1+\frac{N}{p}}(\mathbb{R}_{+}^N))}.
		\end{align}
	Thus, we have
		\begin{align}\label{3.11}
			&\sum_{k=1}^N \left\|P_k^i \partial_i\left( P_k^j \partial_j B\right)\right\|_{L^1(\mathbb{R}_{+};  \dot{B}_{p, 1}^{-1+\frac{N}{p}}(\mathbb{R}_{+}^N))} \nonumber
			\\
			& \lesssim \sum_{r=1}^{N-1}\left\|D^2 u\right\|_{L^1(\mathbb{R}_{+};  \dot{B}_{p, 1}^{-1+\frac{N}{p}}(\mathbb{R}_{+}^N))}^r \left\|\partial_i\left( P_k^j \partial_j B\right)\right\|_{L^1(\mathbb{R}_{+};  \dot{B}_{p, 1}^{-1+\frac{N}{p}}(\mathbb{R}_{+}^N))}\nonumber
			\\
			& \lesssim \sum_{r=1}^{N-1}\left\|D^2 u\right\|_{L^1(\mathbb{R}_{+};  \dot{B}_{p, 1}^{-1+\frac{N}{p}}(\mathbb{R}_{+}^N)t)}^r\sum_{r=1}^{N-1}\left\|D^2 u\right\|_{L^1(\mathbb{R}_{+};  \dot{B}_{p, 1}^{-1+\frac{N}{p}}(\mathbb{R}_{+}^N))}^r  \left\| D^{2} B\right\|_{L^1(\mathbb{R}_{+};  \dot{B}_{p, 1}^{-1+\frac{N}{p}}(\mathbb{R}_{+}^N))}\nonumber
			\\&\lesssim  \sum_{r=2}^{2N-2}\left\|D^2 u\right\|_{L^1(\mathbb{R}_{+};  \dot{B}_{p, 1}^{-1+\frac{N}{p}}(\mathbb{R}_{+}^N))}^r\left\| D^{2} B\right\|_{L^1(\mathbb{R}_{+};  \dot{B}_{p, 1}^{-1+\frac{N}{p}}(\mathbb{R}_{+}^N))}.
		\end{align}
	We shall consider the other terms in $V_{4}(u,B)$ to obtain
		\begin{align}\label{3.12}
			&\left\| B^{i}\partial_{i}u\right\|_{L^1(\mathbb{R}_{+};  \dot{B}_{p, 1}^{-1+\frac{N}{p}}(\mathbb{R}_{+}^N))}+\left\|B^{i}P_{i}^{j}\partial_{j}u\right\|_{L^1(\mathbb{R}_{+};  \dot{B}_{p, 1}^{-1+\frac{N}{p}}(\mathbb{R}_{+}^N))} \nonumber\\ & \lesssim  \left(\sum_{r=1}^{N-1}\left\|\int_0^t D u d \tau\right\|_{L^{\infty}(\mathbb{R}_{+};  \dot{B}_{p, 1}^{\frac{N}{p}}(\mathbb{R}_{+}^N))}^r+1\right)\|B\nabla B\|_{L^1(\mathbb{R}_{+};  \dot{B}_{p, 1}^{-1+\frac{N}{p}}(\mathbb{R}_{+}^N))}+\|B\nabla u\|_{L^1(\mathbb{R}_{+};  \dot{B}_{p, 1}^{-1+\frac{N}{p}}(\mathbb{R}_{+}^N))}\nonumber\\
			&\lesssim \left(\sum_{r=1}^{N-1}\left\|D^2 u\right\|_{L^1(\mathbb{R}_{+};  \dot{B}_{p, 1}^{-1+\frac{N}{p}}(\mathbb{R}_{+}^N))}^r+1\right)\left(\|\nabla B\|_{L^1(\mathbb{R}_{+};  \dot{B}_{p, 1}^{\frac{N}{p}}(\mathbb{R}_{+}^N))} \sup _{t>0}\left\|B\right\|_{\dot{B}_{p, 1}^{-1+\frac{N}{p}}(\mathbb{R}_{+}^N)}\right)\nonumber\\
			&\qquad +\left(\|\nabla u\|_{L^1(\mathbb{R}_{+};  \dot{B}_{p, 1}^{\frac{N}{p}}(\mathbb{R}_{+}^N))} \sup _{t>0}\left\|B\right\|_{\dot{B}_{p, 1}^{-1+\frac{N}{p}}(\mathbb{R}_{+}^N)}\right)\nonumber
			\\&\lesssim \left\| \partial_t B \right\|_{L^{1}(\mathbb{R}_{+},\dot{B}_{p, 1}^{-1+\frac{N}{p}}(\mathbb{R}_{+}^N))}\left(\sum_{r=1}^{N-1}\left\|D^2 u\right\|_{L^1(\mathbb{R}_{+};  \dot{B}_{p, 1}^{-1+\frac{N}{p}}(\mathbb{R}_{+}^N))}^r+1\right)\nonumber\\
			&\qquad\left( \left\|D^{2} B\right\|_{L^1(\mathbb{R}_{+};  \dot{B}_{p, 1}^{-1+\frac{N}{p}}(\mathbb{R}_{+}^N))}+\left\|D^{2} u\right\|_{L^1(\mathbb{R}_{+};  \dot{B}_{p, 1}^{-1+\frac{N}{p}}(\mathbb{R}_{+}^N))}\right).
		\end{align}
	By (\ref{3.9})-(\ref{3.12}), we have
		\begin{align}\label{3.13}
			&\left\| V_4(u, B)\right\|_{L^1(\mathbb{R}_{+};  \dot{B}_{p, 1}^{-1+\frac{N}{p}}(\mathbb{R}_{+}^N))}\nonumber
			\\&\lesssim \sum_{r=2}^{2N-2}\left\|D^2 u\right\|_{L^1(\mathbb{R}_{+};  \dot{B}_{p, 1}^{-1+\frac{N}{p}}(\mathbb{R}_{+}^N))}^r \| D^{2} B\|_{L^1(\mathbb{R}_{+};  \dot{B}_{p, 1}^{-1+\frac{N}{p}}(\mathbb{R}_{+}^N)t)}\nonumber\\&\quad+\left\| \partial_t B \right\|_{L^{1}(\mathbb{R}_{+},\dot{B}_{p, 1}^{-1+\frac{N}{p}}(\mathbb{R}_{+}^N))}\left(\sum_{r=1}^{N-1}\left\|D^2 u\right\|_{L^1(\mathbb{R}_{+};  \dot{B}_{p, 1}^{-1+\frac{N}{p}}(\mathbb{R}_{+}^N))}^r+1\right)\nonumber\\&\quad\times\left( \left\|D^{2} B\right\|_{L^1(\mathbb{R}_{+};  \dot{B}_{p, 1}^{-1+\frac{N}{p}}(\mathbb{R}_{+}^N))}+\left\|D^{2} u\right\|_{L^1(\mathbb{R}_{+};  \dot{B}_{p, 1}^{-1+\frac{N}{p}}(\mathbb{R}_{+}^N))}\right).
		\end{align}
		
Now, we consider $V_{5}(u)$, by (\ref{3.3}) and (\ref{3.2}), it yield that
		\begin{align}\label{3.14}
			\left\|\nabla V_{5}(u)\right\|_{L^1(\mathbb{R}_{+};  \dot{B}_{p, 1}^{-1+\frac{N}{p}}(\mathbb{R}_{+}^N))}   & \lesssim \left\|\nabla (P^{j}_{i} \partial_j u^{i})\right\|_{L^1(\mathbb{R}_{+};  \dot{B}_{p, 1}^{-1+\frac{N}{p}}(\mathbb{R}_{+}^N))}\nonumber\\&\lesssim \left\|\nabla (P^{j}_{i}) \partial_j u^{i})\right\|_{L^1(\mathbb{R}_{+};  \dot{B}_{p, 1}^{-1+\frac{N}{p}}(\mathbb{R}_{+}^N))}+\left\|P^{j}_{i} \nabla\partial_j u^{i}\right\|_{L^1(\mathbb{R}_{+};  \dot{B}_{p, 1}^{-1+\frac{N}{p}}(\mathbb{R}_{+}^N))}\nonumber
			\\&\lesssim \sum_{r=2}^{N} \|D^{2} u\|_{L^1(\mathbb{R}_{+};  \dot{B}_{p, 1}^{-1+\frac{N}{p}}(\mathbb{R}_{+}^N))}^{r}.
		\end{align}
	
	By (\ref{1.8}), we can get that $\partial_{t} P_{j}^{i}$ is the product of a  polynomial of degree $(N-2)$ with a constant term and $\nabla u$.  Since $ \partial_j (P^{j}_{i} u^{i})=P^{j}_{i} \partial_j u^{i}$, we have
		\begin{align}\label{3.15}
			\left\|\partial_t(-\Delta)^{-1} \nabla V_{5}(u)\right\|_{L^1(\mathbb{R}_{+};  \dot{B}_{p, 1}^{-1+\frac{N}{p}}(\mathbb{R}_{+}^N))} &\lesssim \left\|\partial_t(-\Delta)^{-1} \nabla  \partial_j (P^{j}_{i}  u^{i})\right\|_{L^1(\mathbb{R}_{+};  \dot{B}_{p, 1}^{-1+\frac{N}{p}}(\mathbb{R}_{+}^N))}\nonumber
			\\ & \lesssim \left\|\partial_t (P^{j}_{i} u^{i})\right\|_{L^1(\mathbb{R}_{+};  \dot{B}_{p, 1}^{-1+\frac{N}{p}}(\mathbb{R}_{+}^N))}.
		\end{align}
	Thus, by Proposition \ref{Aprop1}, we get
		\begin{align}\label{3.16}
			&\left\|\partial_t (P^{j}_{i}  u^{i})\right\|_{L^1(\mathbb{R}_{+};  \dot{B}_{p, 1}^{-1+\frac{N}{p}}(\mathbb{R}_{+}^N))}\nonumber
			\\&\lesssim \left\|\partial_t P^{j}_{i}  u^{i}\right\|_{L^1(\mathbb{R}_{+};  \dot{B}_{p, 1}^{-1+\frac{N}{p}}(\mathbb{R}_{+}^N))}+\left\| P^{j}_{i} \partial_t u^{i}\right\|_{L^1(\mathbb{R}_{+};  \dot{B}_{p, 1}^{-1+\frac{N}{p}}(\mathbb{R}_{+}^N))}\nonumber
			\\
			& \lesssim  \sum_{r=0}^{N-2}\left\|D^2 u\right\|_{L^1(\mathbb{R}_{+};  \dot{B}_{p, 1}^{-1+\frac{N}{p}}(\mathbb{R}_{+}^N))}^r\|u \nabla u\|_{L^1(\mathbb{R}_{+};  \dot{B}_{p, 1}^{-1+\frac{N}{p}}(\mathbb{R}_{+}^N))}\nonumber
			\\&\qquad +\sum_{r=1}^{N-1}\left\|D^2 u\right\|_{L^1(\mathbb{R}_{+};  \dot{B}_{p, 1}^{-1+\frac{N}{p}}(\mathbb{R}_{+}^N))}^r\left\|\partial_t u\right\|_{L^1(\mathbb{R}_{+};  \dot{B}_{p, 1}^{-1+\frac{N}{p}}(\mathbb{R}_{+}^N))}\nonumber
			\\
			& \lesssim \sum_{r=1}^{N-1}\left\|D^2 u\right\|_{L^1(\mathbb{R}_{+};  \dot{B}_{p, 1}^{-1+\frac{N}{p}}(\mathbb{R}_{+}^N))}^r\left\|\partial_t u\right\|_{L^1(\mathbb{R}_{+};  \dot{B}_{p, 1}^{-1+\frac{N}{p}}(\mathbb{R}_{+}^N))} \nonumber\\
			&\quad  + \sum_{r=0}^{N-2}\left\|D^2 u\right\|_{L^1(\mathbb{R}_{+};  \dot{B}_{p, 1}^{-1+\frac{N}{p}}(\mathbb{R}_{+}^N))}^r\left\|\nabla u\right\|_{L^1(\mathbb{R}_{+};  \dot{B}_{p, 1}^{\frac{N}{p}}(\mathbb{R}_{+}^N))} \sup _{t>0}\left\|u\right\|_{\dot{B}_{p, 1}^{-1+\frac{N}{p}}(\mathbb{R}_{+}^N)}\nonumber \\
			&\lesssim  \sum_{r=1}^{N-1}\left\|D^2 u\right\|_{L^1(\mathbb{R}_{+};  \dot{B}_{p, 1}^{-1+\frac{N}{p}}(\mathbb{R}_{+}^N))}^r\left\|\partial_t u\right\|_{L^1(\mathbb{R}_{+};  \dot{B}_{p, 1}^{-1+\frac{N}{p}}(\mathbb{R}_{+}^N))}.
		\end{align}
	Since the similar arguments in (\ref{3.15}) and (\ref{3.16}) can be applied to estimate $V_{6}(u,B)$,  by (\ref{3.8}), we conclude that
		\begin{align}\label{3.17}
			&\left\|\nabla V_6(u, B)\right\|_{L^1(\mathbb{R}_{+};  \dot{B}_{p, 1}^{-1+\frac{N}{p}}(\mathbb{R}_{+}^N))} \nonumber
			\\ & \lesssim \left\|\nabla (P^{k}_{i} \partial_k B^{i})\right\|_{L^1(\mathbb{R}_{+};  \dot{B}_{p, 1}^{-1+\frac{N}{p}}(\mathbb{R}_{+}^N))}\nonumber\\&\lesssim \left\|\nabla (P^{k}_{i}) \partial_k B^{i}\right\|_{L^1(\mathbb{R}_{+};  \dot{B}_{p, 1}^{-1+\frac{N}{p}}(\mathbb{R}_{+}^N))}+\left\|P^{k}_{i} \nabla\partial_k B^{i}\right\|_{L^1(\mathbb{R}_{+};  \dot{B}_{p, 1}^{-1+\frac{N}{p}}(\mathbb{R}_{+}^N))}\nonumber
			\\&\lesssim  \sum_{r=1}^{N-1}\left\|D^2 u\right\|_{L^1(\mathbb{R}_{+};  \dot{B}_{p, 1}^{-1+\frac{N}{p}}(\mathbb{R}_{+}^N))}^r \left\| D^{2} B\right\|_{L^1(\mathbb{R}_{+};  \dot{B}_{p, 1}^{-1+\frac{N}{p}}(\mathbb{R}_{+}^N))} +\left\|\left(\nabla P_{j}^{i} \right) \nabla B\right\|_{L^1(\mathbb{R}_{+};  \dot{B}_{p, 1}^{-1+\frac{N}{p}}(\mathbb{R}_{+}^N))}\nonumber
			\\& \lesssim \sum_{r=1}^{N-1}\left\|D^2 u\right\|_{L^1(\mathbb{R}_{+};  \dot{B}_{p, 1}^{-1+\frac{N}{p}}(\mathbb{R}_{+}^N))}^r \left\| D^{2} B\right\|_{L^1(\mathbb{R}_{+};  \dot{B}_{p, 1}^{-1+\frac{N}{p}}(\mathbb{R}_{+}^N))},
		\end{align}
	and by Proposition \ref{Aprop1}, we have
		\begin{align}\label{3.18}
			&\left\|\partial_t(-\Delta)^{-1} \nabla V_{6}(u,B)\right\|_{L^1(\mathbb{R}_{+};  \dot{B}_{p, 1}^{-1+\frac{N}{p}}(\mathbb{R}_{+}^N))}\nonumber
			\\ &\lesssim \left\|\partial_t(-\Delta)^{-1} \nabla  \partial_j (P^{j}_{i}  B^{i})\right\|_{L^1(\mathbb{R}_{+};  \dot{B}_{p, 1}^{-1+\frac{N}{p}}(\mathbb{R}_{+}^N))}\nonumber\\ & \lesssim \sum_{r=1}^{N-1}\left\|D^2 u\right\|_{L^1(\mathbb{R}_{+};  \dot{B}_{p, 1}^{-1+\frac{N}{p}}(\mathbb{R}_{+}^N))}^r\left\|\partial_t B\right\|_{L^1t(\mathbb{R}_{+};  \dot{B}_{p, 1}^{-1+\frac{N}{p}}(\mathbb{R}_{+}^N))} \nonumber\\
			& \quad+ \sum_{r=0}^{N-2}\left\|D^2 u\right\|_{L^1(\mathbb{R}_{+};  \dot{B}_{p, 1}^{-1+\frac{N}{p}}(\mathbb{R}_{+}^N))}^r\left\|B \nabla u\right\|_{L^1(\mathbb{R}_{+};  \dot{B}_{p, 1}^{-1+\frac{N}{p}}(\mathbb{R}_{+}^N))}\nonumber\\& \lesssim \sum_{r=1}^{N-1}\left\|D^2 u\right\|_{L^1(\mathbb{R}_{+};  \dot{B}_{p, 1}^{-1+\frac{N}{p}}(\mathbb{R}_{+}^N))}^r\left\|\partial_t B\right\|_{L^1(\mathbb{R}_{+};  \dot{B}_{p, 1}^{-1+\frac{N}{p}}(\mathbb{R}_{+}^N))}\nonumber \\
			&\quad +\sum_{r=0}^{N-2}\left\|D^2 u\right\|_{L^1(\mathbb{R}_{+};  \dot{B}_{p, 1}^{-1+\frac{N}{p}}(\mathbb{R}_{+}^N))}^r\|\nabla u\|_{L^1(\mathbb{R}_{+};  \dot{B}_{p, 1}^{\frac{N}{p}}(\mathbb{R}_{+}^N))} \sup _{t>0}\left\|B\right\|_{\dot{B}_{p, 1}^{-1+\frac{N}{p}}(\mathbb{R}_{+}^N)} \nonumber\\&\lesssim  \sum_{r=1}^{N-1}\left\|D^2 u\right\|_{L^1(\mathbb{R}_{+};  \dot{B}_{p, 1}^{-1+\frac{N}{p}}(\mathbb{R}_{+}^N))}^r\left\|\partial_t B\right\|_{L^1(\mathbb{R}_{+};  \dot{B}_{p, 1}^{-1+\frac{N}{p}}(\mathbb{R}_{+}^N))}.
		\end{align}

	\begin{prop}[cf. \cite{TS}]\label{prop1}
		Let $p\in[N,2N-1)$ and assume that functions $u$ and $q$ satisfy $\partial_t u, D^2 u, \nabla q \in L^1\left(\mathbb{R}_{+};  \dot{B}_{p, 1}^{-1+N / p}(\mathbb{R}_{+}^N)\right)$,  $\left. q\right|_{y_N=0} \in \dot{F}_{1,1}^{1 / 2-1 / 2 p}\left(\mathbb{R}_{+};  \dot{B}_{p, 1}^{-1+N / p}\left(\mathbb{R}^{N-1}\right)\right)$ $\cap$
		$L^1\left(\mathbb{R}_{+};  \dot{B}_{p, 1}^{(N-1) / p}\left(\mathbb{R}^{N-1}\right)\right)$. For the boundary terms $V_{7}(u)$
		and $V_{8} (u,q)$ defined by (\ref{1.12}), the following estimates hold:
			\begin{align*}
				\left\|V_{7} (u)\right\|_{\dot{F}_{1,1}^{\frac{1}{2}-\frac{1}{2 p}} (\mathbb{R}_{+};  \dot{B}_{p, 1}^{-1+\frac{N}{p}}(\mathbb{R}^{N-1}))}
				& \lesssim \sum_{r=2}^{2 N-1}\left(\left\|\partial_t u\right\|_{L^1(\mathbb{R}_{+};  \dot{B}_{p, 1}^{-1+\frac{N}{p}}(\mathbb{R}_{+}^N))}+\left\|D^2 u\right\|_{L^1(\mathbb{R}_{+};  \dot{B}_{p, 1}^{-1+\frac{N}{p}}(\mathbb{R}_{+}^N))}\right)^r,
				\\
				\left\|V_{7} (u)\right\|_{L^1(\mathbb{R}_{+};  \dot{B}_{p, 1}^{\frac{N-1}{p}}(\mathbb{R}^{N-1}))} & \lesssim \sum_{r=2}^{2 N-1}\left\|D^2 u\right\|_{L^1(\mathbb{R}_{+};  \dot{B}_{p, 1}^{-1+\frac{N}{p}}(\mathbb{R}_{+}^N))}^r,
			\end{align*}
				\begin{align*}
			 &\left\|V_{8} (u,q)\right\|_{\dot{F}_{1,1}^{\frac{1}{2}-\frac{1}{2 p}}(\mathbb{R}_{+};  \dot{B}_{p, 1}^{-1+\frac{N}{p}}(\mathbb{R}^{N-1}))}	\nonumber\\
			& \lesssim\left(\left\|\left.q\right|_{y_N=0}\right\|_{\dot{F}_{1,1}^{\frac{1}{2}-\frac{1}{2 p}}(\mathbb{R}_{+};  \dot{B}_{p, 1}^{-1+\frac{N}{p}}(\mathbb{R}^{N-1}))}+\left\|\left.q\right|_{y_N=0}\right\|_{L^1(\mathbb{R}_{+};  \dot{B}_{p, 1}^{-1+\frac{N}{p}}(\mathbb{R}^{N-1}))}\right) \notag\\
			&\quad \times \sum_{r=1}^{N-1}\left(\left\|\partial_t u\right\|_{L^1(\mathbb{R}_{+};  \dot{B}_{p, 1}^{-1+\frac{N}{p}}(\mathbb{R}_{+}^N))}+\left\|D^2 u\right\|_{L^1(\mathbb{R}_{+};  \dot{B}_{p, 1}^{-1+\frac{N}{p}}(\mathbb{R}_{+}^N))}\right)^r,
	\end{align*}
\begin{align*}
			\left\|V_{8} (u,q)\right\|_{L^1(\mathbb{R}_{+};   \dot{B}_{p, 1}^{\frac{N-1}{p}}(\mathbb{R}^{N-1}))}& \lesssim \left\|\left.q\right|_{y_N=0}\right\|_{L^1(\mathbb{R}_{+};  \dot{B}_{p, 1}^{\frac{N-1}{p}}(\mathbb{R}^{N-1})
				)} \sum_{r=1}^{N-1}\left\|D^2 u\right\|_{L^1(\mathbb{R}_{+};  \dot{B}_{p, 1}^{-1+\frac{N}{p}}(\mathbb{R}_{+}^N))}^r.
		\end{align*}
	\end{prop}
	
	\section{Completion of the proof of Theorem  \ref{THM1}}\label{sec4}
	We give an iteration scheme to prove Theorem \ref{THM1} by the Banach fixed point theorem.
	We define an underlying space $\mathbf{U}$ by
	\begin{align*}
		\mathbf{U}= & \bigg\{(u, q, B) :
		u,B \in C\left(\overline{\mathbb{R}_{+}};  \dot{B}_{p, 1}^{-1+\frac{N}{p}}(\mathbb{R}_{+}^N)\right) \cap \dot{W}^{1,1}\left(\mathbb{R}_{+};  \dot{B}_{p, 1}^{-1+\frac{N}{p}}(\mathbb{R}_{+}^N)\right),\\
		&\qquad (D^{2} u, D^{2} B, \nabla q) \in L^1\left(\mathbb{R}_{+};  \dot{B}_{p, 1}^{-1+\frac{N}{p}}(\mathbb{R}_{+}^N)\right), \\
		&\qquad\left.q\right|_{y_N=0} \in \dot{F}_{1,1}^{\frac{1}{2}-\frac{1}{2 p}}\left(\mathbb{R}_{+};  \dot{B}_{p, 1}^{-1+\frac{N}{p}}\left(\mathbb{R}^{N-1}\right)\right) \cap L^1\left(\mathbb{R}_{+};  \dot{B}_{p, 1}^{\frac{N-1}{p}}\left(\mathbb{R}^{N-1}\right)\right), \\
		&\qquad \|(u, q, B)\|_U\leq L\bigg\},
	\end{align*}
	where the metric in $\mathbf{U}$ is induced by the norm
	\begin{align*}
		 \|(u, q, B)\|_{\mathbf{U}} \equiv&\left\|\partial_t u\right\|_{L^1(\mathbb{R}_{+};  \dot{B}_{p, 1}^{-1+\frac{N}{p}}(\mathbb{R}_{+}^N))}+\left\|D^2 u\right\|_{L^1(\mathbb{R}_{+};  \dot{B}_{p, 1}^{-1+\frac{N}{p}}(\mathbb{R}_{+}^N))}+\left\|\partial_t B\right\|_{L^1(\mathbb{R}_{+};  \dot{B}_{p, 1}^{-1+\frac{N}{p}}(\mathbb{R}_{+}^N))}\\
		 &
		+\left\|D^2 B\right\|_{L^1(\mathbb{R}_{+};  \dot{B}_{p, 1}^{-1+\frac{N}{p}}(\mathbb{R}_{+}^N))} 
		+\|\nabla q\|_{L^1(\mathbb{R}_{+};  \dot{B}_{p, 1}^{-1+\frac{N}{p}}(\mathbb{R}_{+}^N))}\\
		& +\left\|\left.q\right|_{y_N=0}\right\|_{\dot{F}_{1,1}^{1 / 2-1 / 2 p}(\mathbb{R}_{+};  \dot{B}_{p, 1}^{-1+\frac{N}{p}}(\mathbb{R}^{N-1}))}+\left\| \left.q\right|_{y_N=0}\right\|_{L^1(\mathbb{R}_{+};  \dot{B}_{p, 1}^{\frac{N-1}{p}}(\mathbb{R}^{N-1}))},
	\end{align*}
	and an absolute constant $L > 0$ is chosen to be sufficiently small, depending on $N$ and $p$.
	
	For each $(\tilde{u},\tilde{B}, \tilde{q})\in \mathbf{U}$, we define the map $\Xi[\cdot]$ through the solution $(u, q, B)$ of
	\begin{equation}\label{4.1}
		\left\{\begin{aligned}
			&\partial_t u-\Delta u+\nabla q=V_{1}(\tilde{u})+V_{2}(\tilde{u}, \tilde{q})+V_{3}(\tilde{u},\tilde{B}),  & &\quad \text{in} \quad \mathbb{R}_{+}\times \mathbb{R}_{+}^N ,
			\\ & \partial_t B-\Delta B=V_{4}(\tilde{u},\tilde{B}), & & \quad \text{in}\quad  \mathbb{R}_{+}\times \mathbb{R}_{+}^N ,
			\\
			& \operatorname{div} u=V_{5}(\tilde{u}), \quad \operatorname{div} B=V_{6}(\tilde{u},\tilde{B}), & & \quad \text{in} \quad \mathbb{R}_{+}\times \mathbb{R}_{+}^N ,
			\\
			&\left(\mathbf{D}(u)-q I\right) \mathbf{n}=V_{7}(\tilde{u})+V_{8}(\tilde{u}, \tilde{q}),& & \quad \text{on}\quad \mathbb{R}_{+}\times \partial \mathbb{R}_{+}^N, \\
			&B=0, & & \quad \text{on}\quad  \mathbb{R}_{+}\times \partial \mathbb{R}_{+}^N,\\
			&u(0, y)=u_0(y), \quad B(0, y)=B_0(y),  & & \quad \text{in}\quad \mathbb{R}_{+}^N,
		\end{aligned}\right.
	\end{equation}
	where
		\begin{align}\label{4.2}
			& V_1(u)=\sum_{k=1}^N \tilde{P}_k^i \partial_i\left( \tilde{P}_k^j \partial_j \tilde{u}\right)+\sum_{j=1}^N\tilde{P}_{j}^{i}\partial_{i}\partial_{j}\tilde{u}
			+\sum_{i=1}^N \partial_{i}(\tilde{P}_{i}^{j}\partial_{j}\tilde{u}), \quad V_{2i}(\tilde{u}, \tilde{q})=-\tilde{P}^{j}_{i} \partial_{j}\tilde{q}, \nonumber\\&  V_3(\tilde{u}, \tilde{B})=\tilde{B}^{i}\partial_{i}\tilde{B}+\tilde{B}^{i}\tilde{P}_{i}^{j}\partial_{j}\tilde{B}, \nonumber\\&
			V_4(\tilde{u}, \tilde{B})=\sum_{k=1}^N \tilde{P}_k^i \partial_i\left( \tilde{P}_k^j \partial_j \tilde{B}\right)+\sum_{j=1}^N\tilde{P}_{j}^{i}\partial_{i}\partial_{j}\tilde{B}
			+\sum_{i=1}^N \partial_{i}(\tilde{P}_{i}^{j}\partial_{j}\tilde{B})+ \tilde{B}^{i}\partial_{i}\tilde{u}+\tilde{B}^{i}\tilde{P}_{i}^{j}\partial_{j}\tilde{u},\nonumber \\
			& V_5(\tilde{u})=\tilde{P}^{j}_{i} \partial_j \tilde{u}^{i}, \quad V_6(\tilde{u}, \tilde{B})= \tilde{P}^{k}_{i} \partial_k \tilde{B}^{i},\nonumber \\
			& V_7(\tilde{u})=-\tilde{A}^{\top} \nabla \tilde{u}+(\nabla \tilde{u})^{\top} \tilde{A}(\tilde{A}-I)^{\top} \mathbf{n}-\left((\tilde{A}-I)^{\top} \nabla \tilde{u}+(\nabla \tilde{u})^{\top}(\tilde{A}-I)\right) \mathbf{n},\nonumber \\
			& V_8(\tilde{u}, \tilde{q})=\tilde{q} I(\tilde{A}-I)^{\top} \mathbf{n},
		\end{align}
	with $\tilde{A}$ and $\tilde{P}$ corresponding to $A$ and $P$,  respectively, in which $u$ is replaced  by $\tilde{u}$. For instance,
	$$
	\tilde{P}_{j}^{i}=P_{j}^{i}\left(\int_0^t \nabla_i^{\prime} \tilde{u}_1 d \tau, \ldots, \int_0^t \nabla_i^{\prime} \tilde{u}_{j-1} d \tau, \int_0^t \nabla_i^{\prime} \tilde{u}_{j+1} d \tau, \ldots, \int_0^t \nabla_i^{\prime} \tilde{u}_N d \tau\right).
	$$
	
	Then, we show that the map given by $\Xi[\tilde{u},\tilde{q},\tilde{B}] = (u,q,B)$ is a contraction mapping from $\mathbf{U}$ to itself.
	
	Applying Theorem \ref{thm2.1} and (\ref{2.16}) to $(u, q, B)$, we obtain
		\begin{align*}
			& \left\|\partial_t u\right\|_{L^1(\mathbb{R}_{+};  \dot{B}_{p, 1}^{-1+\frac{N}{p}}(\mathbb{R}_{+}^N))}+\left\|D^2 u\right\|_{L^1(\mathbb{R}_{+};  \dot{B}_{p, 1}^{-1+\frac{N}{p}}(\mathbb{R}_{+}^N))}+\left\|\partial_t B\right\|_{L^1(\mathbb{R}_{+};  \dot{B}_{p, 1}^{-1+\frac{N}{p}}(\mathbb{R}_{+}^N))}\\
			&+\left\|D^2 B\right\|_{L^1(\mathbb{R}_{+};  \dot{B}_{p, 1}^{-1+\frac{N}{p}}(\mathbb{R}_{+}^N))} +\|\nabla q\|_{L^1(\mathbb{R}_{+};  \dot{B}_{p, 1}^{-1+\frac{N}{p}}(\mathbb{R}_{+}^N))}\\&+\left\|\left.q\right|_{y_N=0}\right\|_{\dot{F}_{1,1}^{1 / 2-1 / 2 p}(\mathbb{R}_{+};  \dot{B}_{p, 1}^{-1+\frac{N}{p}}(\mathbb{R}^{N-1}))}+\left\|\left.q\right|_{y_N=0}\right\|_{L^1(\mathbb{R}_{+};  \dot{B}_{p, 1}^{\frac{N-1}{p}}(\mathbb{R}^{N-1}))}
			\\ & \lesssim\left\|u_0\right\|_{\dot{B}_{p, 1}^{-1+\frac{N}{p}}(\mathbb{R}_{+}^N)}+\left\|B_0\right\|_{\dot{B}_{p, 1}^{-1+\frac{N}{p}}(\mathbb{R}_{+}^N)}
			+\left\|(V_{1}(\tilde{u}),V_{2}(\tilde{u},\tilde{q}),V_{3}(\tilde{u},\tilde{B}),V_{4}(\tilde{u},\tilde{B}))\right\|_{L^1(\mathbb{R}_{+};  \dot{B}_{p, 1}^{-1+\frac{N}{p}}(\mathbb{R}_{+}^N))}
			\\&\quad +\left\|\nabla V_{5}(\tilde{u})\right\|_{L^1(\mathbb{R}_{+};  \dot{B}_{p, 1}^{-1+\frac{N}{p}}(\mathbb{R}_{+}^N))}+\left\|\partial_t(-\Delta)^{-1} \nabla V_{5}(\tilde{u})\right\|_{L^1(\mathbb{R}_{+};  \dot{B}_{p, 1}^{-1+\frac{N}{p}}(\mathbb{R}_{+}^N))}
			\\&\quad +\left\|\nabla V_{6}(\tilde{u},\tilde{B})\right\|_{L^1(\mathbb{R}_{+};  \dot{B}_{p, 1}^{-1+\frac{N}{p}}(\mathbb{R}_{+}^N))}+\left\|\partial_t(-\Delta)^{-1} \nabla V_{6}(\tilde{u},\tilde{B})\right\|_{L^1(\mathbb{R}_{+};  \dot{B}_{p, 1}^{-1+\frac{N}{p}}(\mathbb{R}_{+}^N))}
			\\& \quad +\left\|V_7(\tilde{u})\right\|_{\dot{F}_{1,1}^{\frac{1}{2}-\frac{1}{2 p}}(\mathbb{R}_{+};  \dot{B}_{p, 1}^{-1+\frac{N}{p}}(\mathbb{R}^{N-1}))}+\left\|V_7(\tilde{u})\right\|_{L^1(\mathbb{R}_{+};  \dot{B}_{p, 1}^{\frac{N-1}{p}}(\mathbb{R}^{N-1}))} \\
			& \quad+\left\|V_8(\tilde{u}, \tilde{q})\right\|_{\dot{F}_{1,1}^{\frac{1}{2}-\frac{1}{2 p}}(\mathbb{R}_{+};  \dot{B}_{p, 1}^{-1+\frac{N}{p}}(\mathbb{R}^{N-1}))}+\left\|V_8(\tilde{u}, \tilde{q})\right\|_{L^1(\mathbb{R}_{+};  \dot{B}_{p, 1}^{\frac{N-1}{p}}(\mathbb{R}^{N-1}))}.
		\end{align*}
	Combining (\ref{3.5})-(\ref{3.7}), (\ref{3.13})-(\ref{3.15}), (\ref{3.17}), (\ref{3.18}) with Proposition \ref{prop1}, we have
	\begin{align*}
		& \left\|\partial_t u\right\|_{L^1(\mathbb{R}_{+};  \dot{B}_{p, 1}^{-1+\frac{N}{p}}(\mathbb{R}_{+}^N))}+\left\|D^2 u\right\|_{L^1(\mathbb{R}_{+};  \dot{B}_{p, 1}^{-1+\frac{N}{p}}(\mathbb{R}_{+}^N))}+\left\|\partial_t B\right\|_{L^1(\mathbb{R}_{+};  \dot{B}_{p, 1}^{-1+\frac{N}{p}}(\mathbb{R}_{+}^N))}\\
		&+\left\|D^2 B\right\|_{L^1(\mathbb{R}_{+};  \dot{B}_{p, 1}^{-1+\frac{N}{p}}(\mathbb{R}_{+}^N))} +\|\nabla q\|_{L^1(\mathbb{R}_{+};  \dot{B}_{p, 1}^{-1+\frac{N}{p}}(\mathbb{R}_{+}^N))}\\&+\left\|\left.q\right|_{y_N=0}\right\|_{\dot{F}_{1,1}^{1 / 2-1 / 2 p}(\mathbb{R}_{+};  \dot{B}_{p, 1}^{-1+\frac{N}{p}}(\mathbb{R}^{N-1}))}+\left\|\left.q\right|_{y_N=0}\right\|_{L^1(\mathbb{R}_{+};  \dot{B}_{p, 1}^{\frac{N-1}{p}}(\mathbb{R}^{N-1}))}
		\\ & \lesssim\left\|u_0\right\|_{\dot{B}_{p, 1}^{-1+\frac{N}{p}}(\mathbb{R}_{+}^N)}+\left\|B_0\right\|_{\dot{B}_{p, 1}^{-1+\frac{N}{p}}(\mathbb{R}_{+}^N)}+\sum_{r=2}^{2N-1}L^{r}.
	\end{align*}
	
	Hence, we obtain that	
	\begin{equation}
		\|\Xi[\tilde{u}, \tilde{q},\tilde{B}]\|_\mathbf{U} \lesssim\left\|u_0\right\|_{\dot{B}_{p, 1}^{-1+\frac{N}{p}}(\mathbb{R}_{+}^N)}+\left\|B_0\right\|_{\dot{B}_{p, 1}^{-1+\frac{N}{p}}(\mathbb{R}_{+}^N)}+\sum_{r=2}^{2 N-1} L^{r}.
	\end{equation}
	If we choose $\left\|u_0\right\|_{\dot{B}_{p, 1}^{-1+N / p}}$ and $\left\|B_0\right\|_{\dot{B}_{p, 1}^{-1+N / p}}$ sufficiently small, we conclude that
	$$
	\|\Xi[\tilde{u}, \tilde{q},\tilde{B}]\|_\mathbf{U} \leq L.
	$$
	
	If we establish  that $\Xi$ is a contraction mapping, then by the contraction mapping principle, we conclude the proof of Theorem \ref{THM1}.
	
	Let $(\tilde{u}_{i}, \tilde{q}_{i}, \tilde{B}_{i} ) \in  \mathbf{U}$ ($i = 1,2$). We shall mainly estimate $\|(u_{1}-u_{2}, q_{1}-q_{2}, B_{1}-B_{2} )\|_\mathbf{U}$ with $(u_{i}, q_{i}, B_{i}) = \Xi(\tilde{u}_{i}, \tilde{q}_{i}, \tilde{B}_{i} )$, then we shall prove that $\Xi$ is a contraction mapping on $\mathbf{U}$. We set
	$$
	u^{*}=u_{1}-u_{2}, \quad q^{*}=q_{1}-u_{2}, \quad  B^{*}=B_{1}-B_{2},
	$$
	then by (\ref{4.1}), $u^{*}$, $q^{*}$ and $B^{*}$ satisfy the following equations:
	\begin{equation}\label{4.6}
		\left\{\begin{aligned}
			&\partial_t u^{*}-\Delta u^{*}+\nabla q^{*}=\mathbf{f}_{1},  & &\quad \text{in} \quad \mathbb{R}_{+}\times \mathbb{R}_{+}^N ,
			\\ & \partial_t B^{*}-\Delta B^{*}=\mathbf{f}_{2}, & & \quad \text{in}\quad  \mathbb{R}_{+}\times \mathbb{R}_{+}^N ,
			\\
			& \operatorname{div} u^{*}=\mathrm{f}_{3}, \quad \operatorname{div} B^{*}=\mathrm{f}_{4}, & & \quad \text{in} \quad \mathbb{R}_{+}\times \mathbb{R}_{+}^N ,
			\\
			&\left(\mathbf{D}(u^{*}-q^{*} I\right) \mathbf{n}=\mathbf{f}_{5},& & \quad \text{on}\quad \mathbb{R}_{+}\times \partial \mathbb{R}_{+}^N, \\
			&B^{*}=0, & & \quad \text{on}\quad  \mathbb{R}_{+}\times \partial \mathbb{R}_{+}^N,\\
			&u^{*}(0, y)=0, \quad B^{*}(0, y)=0, & & \quad \text{in}\quad \mathbb{R}_{+}^N,
		\end{aligned}\right.
	\end{equation}
	where
		\begin{align}\label{4.7}
			&\mathbf{f}_{1}=V_{1}(\tilde{u}_{1})+V_{2}(\tilde{u}_{1}, \tilde{q}_{1})+V_{3}(\tilde{u}_{1},\tilde{B}_{1})-V_{1}(\tilde{u}_{2})-V_{2}(\tilde{u}_{2}, \tilde{q}_{2})-V_{3}(\tilde{u}_{2},\tilde{B}_{2}),\nonumber
			\\& \mathbf{f}_{2}=V_{4}(\tilde{u}_{1},\tilde{B}_{1})-V_{4}(\tilde{u}_{2},\tilde{B}_{2}),\quad \mathrm{f}_{3}=V_{5}(\tilde{u}_{1})-V_{5}(\tilde{u}_{2}), \quad \mathrm{f}_{4}=V_{6}(\tilde{u}_{1},\tilde{B}_{1})-V_{6}(\tilde{u}_{2},\tilde{B}_{2}),\nonumber\\
			&\mathbf{f}_{5}=V_{7}(\tilde{u}_{1})+V_{8}(\tilde{u}_{1}, \tilde{q}_{1})-V_{7}(\tilde{u}_{2})-V_{8}(\tilde{u}_{2}, \tilde{q}_{2}).
		\end{align}
	
	We have to estimate the nonlinear terms in (\ref{4.7}). Firstly, let $\tilde{u}_{i,j}$ denote  the $j$-th component of $\tilde{u}_{i}$ with $i=1,2$ and $j=1,\cdots,N$. Let $k=1,2$, we set
	$$
	\left(\int_0^t \nabla_i^{\prime} \tilde{u}_{k,1} d \tau, \ldots, \int_0^t \nabla_i^{\prime} \tilde{u}_{k, j-1}d \tau, \int_0^t \nabla_i^{\prime} \tilde{u}_{k,j+1} d \tau, \ldots, \int_0^t \nabla_i^{\prime} \tilde{u}_{k,N} d \tau\right)=:X_{k,ij}.
	$$
	Since $P_{j}^{i}$ is a polynomial of degree $(N-1)$ without a constant term, we can get
	\begin{equation}\label{4.8}
		P_{j}^{i}(X_{1,ij})-P_{j}^{i}(X_{2,ij})=\int_0^1\left(d_{\mathbf{K}} P_{j}^{i}\right)
		\left(X_{2,ij}+\theta \left(X_{1,ij}-X_{2,ij}\right)\right) d \theta \left(X_{1,ij}-X_{2,ij}\right),
	\end{equation}
	where  $d_{\mathbf{K}}P_{j}^{i}({\mathbf{K}})$ is the derivative of $P_{j}^{i}({\mathbf{K}})$  with respect to $\mathbf{K}$. By (\ref{4.2}), we consider the following formula:
		\begin{align}\label{4.9}
			&\sum_{k=1}^N P_k^i(X_{1,ik})\partial_i\left( P_k^j (X_{1,jk}) \partial_j \tilde{u}_{1}\right)- \sum_{k=1}^N P_k^i(X_{2,ik}) \partial_i\left( P_k^j (X_{2,jk}) \partial_j \tilde{u}_{2}\right) \nonumber
			\\&=\sum_{k=1}^N \bigg(\left(P_k^i(X_{1,ik})-P_k^i(X_{2,ik})\right)\partial_i\left( P_k^j (X_{1,jk}) \partial_j \tilde{u}_{1}\right)\nonumber\\
			&+ P_k^i(X_{2,ik}) \partial_i\left(P_k^j (X_{1,jk}) \partial_j \tilde{u}_{1}-P_k^j (X_{2,jk}) \partial_j \tilde{u}_{2}\right)\bigg)\nonumber
			\\
			&=\sum_{k=1}^N \bigg(\left(P_k^i(X_{1,ik})-P_k^i(X_{2,ik})\right)\partial_i\left( P_k^j (X_{1,jk}) \partial_j \tilde{u}_{1}\right)\nonumber\\
			&+P_k^i(X_{2,ik}) \partial_i\left(\left(P_k^j (X_{1,jk})- P_k^j (X_{2,jk})\right)\partial_j \tilde{u}_{1}-P_k^j (X_{2,jk}) \left(\partial_j \tilde{u}_{1}-\partial_j \tilde{u}_{2}\right)\right)\bigg).
		\end{align}
	From (\ref{3.2}), (\ref{3.3}) and (\ref{4.8}), we have
		\begin{align}\label{4.10}
			&\left\|\left(P_k^i(X_{1,ik})-P_k^i(X_{2,ik})\right)\partial_i\left( P_k^j (X_{1,jk}) \partial_j \tilde{u}_{1}\right)\right\|_{L^1(\mathbb{R}_{+};  \dot{B}_{p, 1}^{-1+\frac{N}{p}}(\mathbb{R}_{+}^N))} \nonumber\\
			&\lesssim
			\left\|\int_0^1\left(d_{\mathbf{K}} P_{k}^{i}\right)
			\left(X_{2,ik}+\theta \left(X_{1,ik}-X_{2,ik}\right)\right) d \theta \partial_i\left( P_k^j (X_{1,jk}) \partial_j \tilde{u}_{1}\right)\right\|_{L^1(\mathbb{R}_{+};  \dot{B}_{p, 1}^{-1+\frac{N}{p}}(\mathbb{R}_{+}^N))}\nonumber\\
			& \qquad\times\left\|X_{1,ik}-X_{2,ik}\right\|_{L^\infty(\mathbb{R}_{+};  \dot{B}_{p, 1}^{\frac{N}{p}}(\mathbb{R}_{+}^N))}\nonumber
			\\&\lesssim \max_{k=1,2}\sum_{r=0}^{N-2}\left\|D^2 \tilde{u}_{k}\right\|_{L^1(\mathbb{R}_{+};  \dot{B}_{p, 1}^{-1+\frac{N}{p}}(\mathbb{R}_{+}^N))}^r\sum_{r=2}^{N}\left\|D^2 \tilde{u}_{1}\right\|_{L^1(\mathbb{R}_{+};  \dot{B}_{p, 1}^{-1+\frac{N}{p}}(\mathbb{R}_{+}^N))}^r \nonumber\\
			&\qquad\times\left\| D^{2}(\tilde{u}_{1}-\tilde{u}_{2})d\tau \right\|_{L^1(\mathbb{R}_{+};  \dot{B}_{p, 1}^{-1+\frac{N}{p}}(\mathbb{R}_{+}^N))}\nonumber
			\\&\lesssim \max_{k=1,2}\sum_{r=2}^{2N-2}\left\|D^2 \tilde{u}_{k}\right\|_{L^1(\mathbb{R}_{+};  \dot{B}_{p, 1}^{-1+\frac{N}{p}}(\mathbb{R}_{+}^N))}^r\left\| D^{2}(\tilde{u}_{1}-\tilde{u}_{2})\right\|_{L^1(\mathbb{R}_{+};  \dot{B}_{p, 1}^{-1+\frac{N}{p}}(\mathbb{R}_{+}^N))}.
		\end{align}
	Then, by (\ref{3.1}),  we have
		\begin{align}\label{4.11}
			&\left\|P_k^i(X_{2,ik}) \partial_i\left(\left(P_k^j (X_{1,jk})- P_k^j (X_{2,jk})\right)\partial_j \tilde{u}_{1}\right)\right\|_{L^1(\mathbb{R}_{+};  \dot{B}_{p, 1}^{-1+\frac{N}{p}}(\mathbb{R}_{+}^N))} \nonumber \\&\lesssim  \left\| \partial_i\left(\left(P_k^j (X_{1,jk})- P_k^j (X_{2,jk})\right)\partial_j \tilde{u}_{1}\right)\right\|_{L^1(\mathbb{R}_{+};  \dot{B}_{p, 1}^{-1+\frac{N}{p}}(\mathbb{R}_{+}^N))}\sum_{r=1}^{N-1}\left\|D^2 \tilde{u}_{2}\right\|_{L^1(\mathbb{R}_{+};  \dot{B}_{p, 1}^{-1+\frac{N}{p}}(\mathbb{R}_{+}^N))}^r.
		\end{align}
	By Proposition \ref{Aprop} and (\ref{4.8}), we have
	\begin{align}
		&\left\|\partial_{i} (P_k^j (X_{1,jk})-P_k^j (X_{2,jk}))\partial_{j}\tilde{u}_{1}\right\|_{L^1(\mathbb{R}_{+};  \dot{B}_{p, 1}^{-1+\frac{N}{p}}(\mathbb{R}_{+}^N))}
		\notag\\&\lesssim \left\|\partial_{i}\left( \int_0^1\left(d_{\mathbf{K}} P_{k}^{j}\right)
		\left(X_{2,jk}+\theta \left(X_{1,jk}-X_{2,jk}\right)\right) d \theta \left(X_{1,jk}-X_{2,jk}\right)\right)\partial_{j}\tilde{u}_{1}\right\|_{L^1(\mathbb{R}_{+};  \dot{B}_{p, 1}^{-1+\frac{N}{p}}(\mathbb{R}_{+}^N))}
		\notag\\&\lesssim \max_{k=1,2}\sum_{r=0}^{N-3}\left\|D^2 \tilde{u}_{k}\right\|_{L^1(\mathbb{R}_{+};  \dot{B}_{p, 1}^{-1+\frac{N}{p}}(\mathbb{R}_{+}^N))}^r
		\max_{k=1,2}\left\|D^2 \tilde{u}_{k}\right\|_{L^1(\mathbb{R}_{+};  \dot{B}_{p, 1}^{-1+\frac{N}{p}}(\mathbb{R}_{+}^N))}
		\left\|\nabla\tilde{u}_{1} \right\|_{L^1(\mathbb{R}_{+};  \dot{B}_{p, 1}^{\frac{N}{p}}(\mathbb{R}_{+}^N))}\nonumber\\
		&\qquad\times\left\|X_{1,jk}-X_{2,jk}\right\|_{L^\infty(\mathbb{R}_{+};  \dot{B}_{p, 1}^{\frac{N}{p}}(\mathbb{R}_{+}^N))}
		\notag\\&\quad+\max_{k=1,2}\sum_{r=0}^{N-2}\left\|D^2 \tilde{u}_{k}\right\|_{L^1(\mathbb{R}_{+};  \dot{B}_{p, 1}^{-1+\frac{N}{p}}(\mathbb{R}_{+}^N))}^r
		\left\|\partial_{i}(X_{1,jk}-X_{2,jk}) \right\|_{L^\infty(\mathbb{R}_{+};  \dot{B}_{p, 1}^{-1+\frac{N}{p}}(\mathbb{R}_{+}^N))}\nonumber\\
		&\qquad\times
		\left\|\nabla\tilde{u}_{1} \right\|_{L^1(\mathbb{R}_{+};  \dot{B}_{p, 1}^{\frac{N}{p}}(\mathbb{R}_{+}^N))}
		\notag\\&\lesssim \max_{k=1,2}\sum_{r=1}^{N-1}\left\|D^2 \tilde{u}_{k}\right\|_{L^1(\mathbb{R}_{+};  \dot{B}_{p, 1}^{-1+\frac{N}{p}}(\mathbb{R}_{+}^N))}^r\left\| D^{2}(\tilde{u}_{1}-\tilde{u}_{2})\right\|_{L^1(\mathbb{R}_{+};  \dot{B}_{p, 1}^{-1+\frac{N}{p}}(\mathbb{R}_{+}^N))},\label{4.12}
	\end{align}
	and
		\begin{align}\label{4.13}
			&\left\| (P_k^j (X_{1,jk})-P_k^j (X_{2,jk}))\partial_{i}\partial_{j}\tilde{u}_{1}\right\|_{L^1(\mathbb{R}_{+};  \dot{B}_{p, 1}^{-1+\frac{N}{p}}(\mathbb{R}_{+}^N))} \nonumber
			\\ & \lesssim
			\left\|\int_0^1\left(d_{\mathbf{K}} P_{k}^{j}\right)
			\left(X_{2,jk}+\theta \left(X_{1,jk}-X_{2,jk}\right)\right) d \theta \partial_i \partial_j \tilde{u}_{1}\right\|_{L^1(\mathbb{R}_{+};  \dot{B}_{p, 1}^{-1+\frac{N}{p}}(\mathbb{R}_{+}^N))}\nonumber\\
			&\qquad\times \left\|X_{1,ik}-X_{2,ik}\right\|_{L^\infty(\mathbb{R}_{+};  \dot{B}_{p, 1}^{\frac{N}{p}}(\mathbb{R}_{+}^N))}\nonumber
			\\&\lesssim \max_{k=1,2}\sum_{r=1}^{N-1}\left\|D^2 \tilde{u}_{k}\right\|_{L^1(\mathbb{R}_{+};  \dot{B}_{p, 1}^{-1+\frac{N}{p}}(\mathbb{R}_{+}^N))}^r\left\| D^{2}(\tilde{u}_{1}-\tilde{u}_{2})\right\|_{L^1(\mathbb{R}_{+};  \dot{B}_{p, 1}^{-1+\frac{N}{p}}(\mathbb{R}_{+}^N))}.
		\end{align}
	
	Then, from (\ref{3.2}) and (\ref{3.3}), we get that
		\begin{align}\label{4.14}
			&\left\|P_k^i(X_{2,ik})\partial_{i}\left(P_k^j (X_{2,jk}) \partial_j (\tilde{u}_{1}- \tilde{u}_{2})\right)\right\|_{L^1(\mathbb{R}_{+};  \dot{B}_{p, 1}^{-1+\frac{N}{p}}(\mathbb{R}_{+}^N))} \nonumber\\&\lesssim \sum_{r=1}^{N-1}\left\|D^2 \tilde{u}_{2}\right\|_{L^1(\mathbb{R}_{+};  \dot{B}_{p, 1}^{-1+\frac{N}{p}}(\mathbb{R}_{+}^N))}^r\left\| \partial_i\left(P_k^j (X_{2,jk})\partial_j (\tilde{u}_{1}- \tilde{u}_{2})\right)\right\|_{L^1(\mathbb{R}_{+};  \dot{B}_{p, 1}^{-1+\frac{N}{p}}(\mathbb{R}_{+}^N))}\nonumber
			\\ & \lesssim  \sum_{r=1}^{N-1}\left\|D^2 \tilde{u}_{2}\right\|_{L^1(\mathbb{R}_{+};  \dot{B}_{p, 1}^{-1+\frac{N}{p}}(\mathbb{R}_{+}^N))}^r \left(\sum_{r=1}^{N-1}\left\|D^2 \tilde{u}_{2}\right\|_{L^1(\mathbb{R}_{+};  \dot{B}_{p, 1}^{-1+\frac{N}{p}}(\mathbb{R}_{+}^N))}^r\left\|\nabla (\tilde{u}_{1}-\tilde{u}_{2})\right\|_{L^1(\mathbb{R}_{+};  \dot{B}_{p, 1}^{\frac{N}{p}}(\mathbb{R}_{+}^N))}\right.\nonumber\\&\quad \times
			\left.\sum_{r=1}^{N-1}\left\|D^2 \tilde{u}_{2}\right\|_{L^1(\mathbb{R}_{+};  \dot{B}_{p, 1}^{-1+\frac{N}{p}}(\mathbb{R}_{+}^N))}^r \left\| D^{2}(\tilde{u}_{1}-\tilde{u}_{2})\right\|_{L^1(\mathbb{R}_{+};  \dot{B}_{p, 1}^{-1+\frac{N}{p}}(\mathbb{R}_{+}^N))}\right)\nonumber
			\\& \lesssim \sum_{r=1}^{2N-2}\left\|D^2 \tilde{u}_{2}\right\|_{L^1(\mathbb{R}_{+};  \dot{B}_{p, 1}^{-1+\frac{N}{p}}(\mathbb{R}_{+}^N))}^r\left\| D^{2}(\tilde{u}_{1}-\tilde{u}_{2})\right\|_{L^1(\mathbb{R}_{+};  \dot{B}_{p, 1}^{-1+\frac{N}{p}}(\mathbb{R}_{+}^N))}.
		\end{align}
	
	Thus, by (\ref{4.9})-(\ref{4.14}), we have
		\begin{align}\label{4.15}
			&\left\|V_{1}(\tilde{u}_{1})-V_{1}(\tilde{u}_{2})\right\|_{L^1(\mathbb{R}_{+};  \dot{B}_{p, 1}^{-1+\frac{N}{p}}(\mathbb{R}_{+}^N))}\nonumber\\
			&\lesssim \max_{k=1,2}\sum_{r=1}^{2N-2}\left\|D^2 \tilde{u}_{k}\right\|_{L^1(\mathbb{R}_{+};  \dot{B}_{p, 1}^{-1+\frac{N}{p}}(\mathbb{R}_{+}^N))}^r\left\| D^{2}(\tilde{u}_{1}-\tilde{u}_{2})\right\|_{L^1(\mathbb{R}_{+};  \dot{B}_{p, 1}^{-1+\frac{N}{p}}(\mathbb{R}_{+}^N))}.
		\end{align}
	
	Employing the same argument as that in proving (\ref{4.15}), we can get
		\begin{align}\label{4.16}
			&\left\|V_{2}(\tilde{u}_{1},\tilde{q}_{1})-V_{2}(\tilde{u}_{2},\tilde{q}_{2})\right\|_{L^1(\mathbb{R}_{+};  \dot{B}_{p, 1}^{-1+\frac{N}{p}}(\mathbb{R}_{+}^N))} \nonumber
			\\&\lesssim\left\|\left(P_i^j(X_{1,ij})-P_i^j(X_{2,ij})\right)\partial_j\tilde{q}_{1}\right\|_{L^1(\mathbb{R}_{+};  \dot{B}_{p, 1}^{-1+\frac{N}{p}}(\mathbb{R}_{+}^N))}+ \left\|P_i^j(X_{2,ij}) \partial_j\left( \tilde{q}_{1}-\tilde{q}_{2}\right)\right\|_{L^1(\mathbb{R}_{+};  \dot{B}_{p, 1}^{-1+\frac{N}{p}}(\mathbb{R}_{+}^N))}\nonumber
			\\&\lesssim \max_{k=1,2}\sum_{r=1}^{2N-2}\left\|D^2 \tilde{u}_{k}\right\|_{L^1(\mathbb{R}_{+};  \dot{B}_{p, 1}^{-1+\frac{N}{p}}(\mathbb{R}_{+}^N))}^r\nonumber\\
			&\qquad\times\bigg(\left\| D^{2}(\tilde{u}_{1}-\tilde{u}_{2})\right\|_{L^1(\mathbb{R}_{+};  \dot{B}_{p, 1}^{-1+\frac{N}{p}}(\mathbb{R}_{+}^N))}+\left\| \nabla(\tilde{q}_{1}-\tilde{q}_{2})\right\|_{L^1(\mathbb{R}_{+};  \dot{B}_{p, 1}^{-1+\frac{N}{p}}(\mathbb{R}_{+}^N))}\bigg).
		\end{align}
	Notice that
	\begin{equation}\label{4.17}
		\begin{aligned}
			&V_{3}(\tilde{u}_{1},\tilde{B}_{1})-V_{3}(\tilde{u}_{2},\tilde{B}_{2})
			\\&=\left(\tilde{B}_{1}^{i}-\tilde{B}_{2}^{i}\right)\partial_{i}\tilde{B}_{1}+\tilde{B}_{2}^{i}\partial_{i}
			\left(\tilde{B}_{1}-\tilde{B}_{2}\right)+\left(\tilde{B}_{1}^{i}-\tilde{B}_{2}^{i}\right)P_{i}^{j}(X_{1,ij})\partial_{j}\tilde{B}_{1}\nonumber\\
			&\qquad
			+\tilde{B}_{2}^{i}\left(P_{i}^{j}(X_{1,ij})\partial_{j}\tilde{B}_{1}-P_{i}^{j}(X_{2,ij})\partial_{j}\tilde{B}_{2}
			\right).
		\end{aligned}
	\end{equation}
	By Propositions \ref{Aprop1} and \ref{Aprop}, we have
		\begin{align}\label{4.18}
			&\left\|\left(\tilde{B}_{1}^{i}-\tilde{B}_{2}^{i}\right)\partial_{i}\tilde{B}_{1}\right\|_{L^1(\mathbb{R}_{+};  \dot{B}_{p, 1}^{-1+\frac{N}{p}}(\mathbb{R}_{+}^N))}+\left\|\tilde{B}_{2}^{i}\partial_{i}
			\left(\tilde{B}_{1}-\tilde{B}_{2}\right)\right\|_{L^1(\mathbb{R}_{+};  \dot{B}_{p, 1}^{-1+\frac{N}{p}}(\mathbb{R}_{+}^N))} \nonumber
			\\ & \lesssim \left\|\nabla\tilde{B}_{1}\right\|_{L^1(\mathbb{R}_{+};  \dot{B}_{p, 1}^{\frac{N}{p}}(\mathbb{R}_{+}^N))}
			\sup_{t>0}\left\|\tilde{B}_{1}-\tilde{B}_{2}\right\|_{\dot{B}_{p, 1}^{-1+\frac{N}{p}}(\mathbb{R}_{+}^N)}\nonumber\\
			&\qquad
			+\left\|\nabla\left(\tilde{B}_{1}-\tilde{B}_{2}\right)\right\|_{L^1(\mathbb{R}_{+};  \dot{B}_{p, 1}^{\frac{N}{p}}(\mathbb{R}_{+}^N))}
			\sup_{t>0}\left\|\tilde{B}_{2}\right\|_{\dot{B}_{p, 1}^{-1+\frac{N}{p}}(\mathbb{R}_{+}^N)}\nonumber
			\\ & \lesssim  \left\|D^{2}\tilde{B}_{1}\right\|_{L^1(\mathbb{R}_{+};  \dot{B}_{p, 1}^{-1+\frac{N}{p}}(\mathbb{R}_{+}^N))}
			\left\|\partial_{t}\left(\tilde{B}_{1}-\tilde{B}_{2}\right)\right\|_{L^1(\mathbb{R}_{+};  \dot{B}_{p, 1}^{-1+\frac{N}{p}}(\mathbb{R}_{+}^N))}\nonumber\\
			&\qquad
		 +\left\|D^{2}\left(\tilde{B}_{1}-\tilde{B}_{2}\right)\right\|_{L^1(\mathbb{R}_{+};  \dot{B}_{p, 1}^{-1+\frac{N}{p}}(\mathbb{R}_{+}^N))}
			\left\|\partial_{t}\tilde{B}_{2}\right\|_{L^1(\mathbb{R}_{+};  \dot{B}_{p, 1}^{-1+\frac{N}{p}}(\mathbb{R}_{+}^N))}.
		\end{align}
	Noting that (\ref{4.17}) has the similar formula to that in (\ref{4.9}), we have that
	\begin{align}
		&\left\|\left(\tilde{B}_{1}^{i}-\tilde{B}_{2}^{i}\right)P_{i}^{j}(X_{1,ij})\partial_{j}\tilde{B}_{1}\right\|_
		{L^1(\mathbb{R}_{+};  \dot{B}_{p, 1}^{-1+\frac{N}{p}}(\mathbb{R}_{+}^N))}
		\notag\\ & \lesssim  \sum_{r=1}^{N-1}\left\|D^2 \tilde{u}_{1}\right\|_{L^1(\mathbb{R}_{+};  \dot{B}_{p, 1}^{-1+\frac{N}{p}}(\mathbb{R}_{+}^N))}^r\left\| \left(\tilde{B}_{1}-\tilde{B}_{2}\right) \nabla \tilde{B}_{1}\right\|_{L^1(\mathbb{R}_{+};  \dot{B}_{p, 1}^{-1+\frac{N}{p}}(\mathbb{R}_{+}^N))}
		\notag\\&\lesssim  \sum_{r=1}^{N-1}\left\|D^2 \tilde{u}_{1}\right\|_{L^1(\mathbb{R}_{+};  \dot{B}_{p, 1}^{-1+\frac{N}{p}}(\mathbb{R}_{+}^N))}^r\left\|D^{2} \tilde{B}_{1}\right\|_{L^1(\mathbb{R}_{+};  \dot{B}_{p, 1}^{-1+\frac{N}{p}}(\mathbb{R}_{+}^N))} \left\| \partial_{t}\left(\tilde{B}_{1}-\tilde{B}_{2}\right)\right\|_{L^1(\mathbb{R}_{+};  \dot{B}_{p, 1}^{-1+\frac{N}{p}}(\mathbb{R}_{+}^N))},\label{4.19}
	\end{align}
	and
		\begin{align}\label{4.20}
			&\left\|\tilde{B}_{2}^{i}\left(P_{i}^{j}(X_{1,ij})\partial_{j}\tilde{B}_{1}-P_{i}^{j}(X_{2,ij})\partial_{j}\tilde{B}_{2}\right)\right\|_
			{L^1(\mathbb{R}_{+};  \dot{B}_{p, 1}^{-1+\frac{N}{p}}(\mathbb{R}_{+}^N))}\nonumber
			\\& \lesssim \left\|\tilde{B}_{2}^{i}P_{i}^{j}(X_{1,ij})\left(\partial_{j}\tilde{B}_{1}- \partial_{j}\tilde{B}_{2}\right)\right\|_
			{L^1(\mathbb{R}_{+};  \dot{B}_{p, 1}^{-1+\frac{N}{p}}(\mathbb{R}_{+}^N))}\nonumber\\
			&\qquad+\left\|\tilde{B}_{2}^{i}\left(P_{i}^{j}(X_{1,ij})-P_{i}^{j}(X_{2,ij})\right)\partial_{j}\tilde{B}_{2}\right\|_
			{L^1(\mathbb{R}_{+};  \dot{B}_{p, 1}^{-1+\frac{N}{p}}(\mathbb{R}_{+}^N))}\nonumber
			\\& \lesssim \sum_{r=1}^{N-1}\left\|D^2 \tilde{u}_{1}\right\|_{L^1(\mathbb{R}_{+};  \dot{B}_{p, 1}^{-1+\frac{N}{p}}(\mathbb{R}_{+}^N))}^r\left\| \nabla \left(\tilde{B}_{1}-\tilde{B}_{2}\right)\tilde{B}_{2}\right\|_{L^1(\mathbb{R}_{+};  \dot{B}_{p, 1}^{-1+\frac{N}{p}}(\mathbb{R}_{+}^N))}\nonumber
			\\&\quad +\max_{k=1,2}\sum_{r=0}^{N-2}\left\|D^2 \tilde{u}_{k}\right\|_{L^1(\mathbb{R}_{+};  \dot{B}_{p, 1}^{-1+\frac{N}{p}}(\mathbb{R}_{+}^N))}^r\left\|\nabla\tilde{B}_{2}\tilde{B}_{2} \right\|_{L^1(\mathbb{R}_{+};  \dot{B}_{p, 1}^{-1+\frac{N}{p}}(\mathbb{R}_{+}^N))}\nonumber\\
			&\qquad \times \left\|X_{1,ik}-X_{2,ik}\right\|_{L^\infty(\mathbb{R}_{+};  \dot{B}_{p, 1}^{\frac{N}{p}}(\mathbb{R}_{+}^N))}\nonumber
			\\&\lesssim \max_{k=1,2}\sum_{r=0}^{N-1}\left\|D^2 \tilde{u}_{k}\right\|_{L^1(\mathbb{R}_{+};  \dot{B}_{p, 1}^{-1+\frac{N}{p}}(\mathbb{R}_{+}^N))}^{r}\left\|\partial_{t} \tilde{B}_{2}\right\|_{L^1(\mathbb{R}_{+};  \dot{B}_{p, 1}^{-1+\frac{N}{p}}(\mathbb{R}_{+}^N))}\nonumber
			\\ & \quad \times \bigg(\left\|D^{2} \left(\tilde{B}_{1}-\tilde{B}_{2}\right)\right\|_{L^1(\mathbb{R}_{+};  \dot{B}_{p, 1}^{-1+\frac{N}{p}}(\mathbb{R}_{+}^N))}\nonumber\\
			&\qquad\qquad+
			\left\|D^{2} \tilde{B}_{2}\right\|_{L^1(\mathbb{R}_{+};  \dot{B}_{p, 1}^{-1+\frac{N}{p}}(\mathbb{R}_{+}^N))} \left\|D^{2}\left(\tilde{u}_{1}-\tilde{u}_{2}\right)\right\|_{L^1(\mathbb{R}_{+};  \dot{B}_{p, 1}^{-1+\frac{N}{p}}(\mathbb{R}_{+}^N))}\bigg).
		\end{align}

	In conclusion,  we can obtain that
	\begin{equation}\label{4.21}
		\left\|\mathbf{f}_{1}\right\|_{L^1(\mathbb{R}_{+};  \dot{B}_{p, 1}^{-1+\frac{N}{p}}(\mathbb{R}_{+}^N))}\lesssim\sum_{r=1}^{2 N-2} L^r\left\|
		\tilde{u}_{1}-\tilde{u}_{2},\tilde{q}_{1}-\tilde{q}_{2},\tilde{B}_{1}-\tilde{B}_{2}\right\|_{\mathbf{U}}.
	\end{equation}
	
	Employing the same argument as that in proving (\ref{4.21}), we can obtain
	\begin{equation}\label{4.22}
		\left\|\mathbf{f}_{2}\right\|_{L^1(\mathbb{R}_{+};  \dot{B}_{p, 1}^{-1+\frac{N}{p}}(\mathbb{R}_{+}^N))}\lesssim\sum_{r=1}^{2 N-2} L^r\left\|
		\tilde{u}_{1}-\tilde{u}_{2},\tilde{B}_{1}-\tilde{B}_{2}\right\|_{\mathbf{U}}.
	\end{equation}
	
	We now consider $\mathrm{f}_{3}$ and $\mathrm{f}_{4}$,  since  $\mathrm{f}_{3}$ has the similar formula to that in $\mathrm{f}_{4}$, then we just need to estimate $\mathrm{f}_{4}$.  In view of (\ref{4.2}), we have
	\begin{align*}
		\nabla \mathrm{f}_{4}&=\nabla\left(P_{i}^{k}(X_{1,ik})\partial_{k}\tilde{B}_{1}^{i}- P_{i}^{k}(X_{2,ik})\partial_{k}\tilde{B}_{2}^{i}\right)\nonumber\\
		&
	=\nabla\left(\left(P_{i}^{k}(X_{1,ik})-P_{i}^{k}(X_{2,ik})\right)\partial_{k}\tilde{B}_{1}^{i}+
		P_{i}^{k}(X_{2,ik})\left(\partial_{k}\tilde{B}_{1}^{i}-\partial_{k}\tilde{B}_{2}^{i}\right)\right).
	\end{align*}
	By (\ref{3.8}), (\ref{3.9}), (\ref{4.12}) and (\ref{4.13}), we have
		\begin{align}\label{4.23}
			&\left\|\nabla \mathrm{f}_{4}\right\|_{L^1(\mathbb{R}_{+};  \dot{B}_{p, 1}^{-1+\frac{N}{p}}(\mathbb{R}_{+}^N))}  \nonumber\\
		& \lesssim \left\|\nabla\left(\left(P_{i}^{k}(X_{1,ik})-P_{i}^{k}(X_{2,ik})\right)\partial_{k}\tilde{B}_{1}^{i}\right)\right\|_{L^1(\mathbb{R}_{+};  \dot{B}_{p, 1}^{-1+\frac{N}{p}}(\mathbb{R}_{+}^N))}\nonumber\\
		&\qquad+\left\|\nabla\left(P_{i}^{k}(X_{2,ik})\left(\partial_{k}\tilde{B}_{1}^{i}-\partial_{k}\tilde{B}_{2}^{i}\right)\right)\right\|_{L^1(\mathbb{R}_{+};  \dot{B}_{p, 1}^{-1+\frac{N}{p}}(\mathbb{R}_{+}^N))}\nonumber
			\\&\lesssim \max_{k=1,2}\sum_{r=0}^{N-3}\left\|D^2 \tilde{u}_{k}\right\|_{L^1(\mathbb{R}_{+};  \dot{B}_{p, 1}^{-1+\frac{N}{p}}(\mathbb{R}_{+}^N))}^r
			\max_{k=1,2}\left\|D^2 \tilde{u}_{k}\right\|_{L^1(\mathbb{R}_{+};  \dot{B}_{p, 1}^{-1+\frac{N}{p}}(\mathbb{R}_{+}^N))}\nonumber\\
			&\qquad\times\left\|\nabla\tilde{B}_{1} \right\|_{L^1(\mathbb{R}_{+};  \dot{B}_{p, 1}^{\frac{N}{p}}(\mathbb{R}_{+}^N))}\left\|X_{1,jk}-X_{2,jk}\right\|_{L^\infty(\mathbb{R}_{+};  \dot{B}_{p, 1}^{\frac{N}{p}}(\mathbb{R}_{+}^N))}\nonumber
			\\&\quad+\max_{k=1,2}\sum_{r=0}^{N-2}\left\|D^2 \tilde{u}_{k}\right\|_{L^1(\mathbb{R}_{+};  \dot{B}_{p, 1}^{-1+\frac{N}{p}}(\mathbb{R}_{+}^N))}^r
			\left\|\partial_{i}(X_{1,jk}-X_{2,jk}) \right\|_{L^\infty(\mathbb{R}_{+};  \dot{B}_{p, 1}^{-1+\frac{N}{p}}(\mathbb{R}_{+}^N))}\nonumber\\
			&\qquad\times\left\|\nabla\tilde{B}_{1} \right\|_{L^1(\mathbb{R}_{+};  \dot{B}_{p, 1}^{\frac{N}{p}}(\mathbb{R}_{+}^N))}\nonumber
			\\&\quad  +\max_{k=1,2}\sum_{r=0}^{N-2}\left\|D^2 \tilde{u}_{k}\right\|_{L^1(\mathbb{R}_{+};  \dot{B}_{p, 1}^{-1+\frac{N}{p}}(\mathbb{R}_{+}^N))}^r \left\|D^2 \tilde{B}_{1}\right\|_{L^1(\mathbb{R}_{+};  \dot{B}_{p, 1}^{-1+\frac{N}{p}}(\mathbb{R}_{+}^N))}\nonumber\\
			&\qquad\times
			\left\|X_{1,jk}-X_{2,jk}\right\|_{L^\infty(\mathbb{R}_{+};  \dot{B}_{p, 1}^{\frac{N}{p}}(\mathbb{R}_{+}^N))} \nonumber
			\\&\quad +\sum_{r=1}^{N-1}\left\|D^2 \tilde{u}_{2}\right\|_{L^1(\mathbb{R}_{+};  \dot{B}_{p, 1}^{-1+\frac{N}{p}}(\mathbb{R}_{+}^N))}^r+\left\|D^2 (\tilde{B}_{1}- \tilde{B}_{2})\right\|_{L^1(\mathbb{R}_{+};  \dot{B}_{p, 1}^{-1+\frac{N}{p}}(\mathbb{R}_{+}^N))} \nonumber
			\\&\lesssim \sum_{r=1}^{2 N-2} L^r\left\|
			\tilde{u}_{1}-\tilde{u}_{2},\tilde{B}_{1}-\tilde{B}_{2}\right\|_{\mathbf{U}},
		\end{align}
	and  by (\ref{3.18}) and (\ref{4.8}), it yields that
		\begin{align}\label{4.24}
			&\left\|\partial_t(-\Delta)^{-1} \nabla \mathrm{f}_{4}\right\|_{L^1(\mathbb{R}_{+};  \dot{B}_{p, 1}^{-1+\frac{N}{p}}(\mathbb{R}_{+}^N))}\nonumber\\
			&\lesssim \left\|\partial_t(-\Delta)^{-1} \nabla  \partial_k \left(P^{k}_{i} (X_{1,ik})\tilde{B}^{i}_{1}-P^{k}_{i} (X_{2,ik})\tilde{B}^{i}_{2}\right)\right\|_{L^1(\mathbb{R}_{+};  \dot{B}_{p, 1}^{-1+\frac{N}{p}}(\mathbb{R}_{+}^N))}\nonumber
			\\&\lesssim \left\|\partial_t \left(\left(P^{k}_{i} (X_{1,ik})-P^{k}_{i} (X_{2,ik})\right)\tilde{B}^{i}_{1}\right)\right\|_{L^1(\mathbb{R}_{+};  \dot{B}_{p, 1}^{-1+\frac{N}{p}}(\mathbb{R}_{+}^N))}\nonumber\\
			&\qquad +\left\| \partial_t \left(P^{k}_{i} (X_{2,ik})\left(\tilde{B}^{i}_{1}-\tilde{B}^{i}_{2}\right)\right)\right\|_{L^1(\mathbb{R}_{+};  \dot{B}_{p, 1}^{-1+\frac{N}{p}}(\mathbb{R}_{+}^N))}\nonumber
			\\& \lesssim \sum_{r=1}^{2 N-2} L^r\left\|
			\tilde{u}_{1}-\tilde{u}_{2},\tilde{B}_{1}-\tilde{B}_{2}\right\|_{\mathbf{U}}.
		\end{align}

	The similar arguments to (\ref{4.23}) and (\ref{4.24}) can be applicable for estimating $\mathrm{f}_{3}$ to get
		\begin{align}\label{4.25}
			&\left\|\nabla \mathbf{f}_{3}\right\|_{L^1(\mathbb{R}_{+};  \dot{B}_{p, 1}^{-1+\frac{N}{p}}(\mathbb{R}_{+}^N))} +\left\|\partial_t(-\Delta)^{-1} \nabla \mathbf{f}_{3}\right\|_{L^1(\mathbb{R}_{+};  \dot{B}_{p, 1}^{-1+\frac{N}{p}}(\mathbb{R}_{+}^N))} \lesssim \sum_{r=1}^{2 N-2} L^r\left\|
			\tilde{u}_{1}-\tilde{u}_{2}\right\|_{\mathbf{U}}.
		\end{align}

	We next consider $\mathbf{f}_{5}$ given in (\ref{4.7}), while the detailed proof was given in \cite{TS}, we have
	\begin{equation}\label{4.26}
		\begin{aligned}
			&\left\|\mathbf{f}_{5}\right\|_{\dot{F}_{1,1}^{\frac{1}{2}-\frac{1}{2 p}} (\mathbb{R}_{+};  \dot{B}_{p, 1}^{-1+\frac{N}{p}}(\mathbb{R}^{N-1}))}
			+\left\|\mathbf{f}_{5}\right\|_{L^1(\mathbb{R}_{+};  \dot{B}_{p, 1}^{\frac{N-1}{p}}(\mathbb{R}^{N-1}))}
			\lesssim \sum_{r=1}^{2 N-2} L^r\left\|
			\tilde{u}_{1}-\tilde{u}_{2},\tilde{q}_{1}-\tilde{q}_{2}, \tilde{B}_{1}-\tilde{B}_{2}\right\|_{\mathbf{U}}.
		\end{aligned}
	\end{equation}
	
	Therefore, if we choose
	$$
	C \sum_{k=1}^{2 N-2} L^k \leq \frac{1}{2},
	$$
	then  by (\ref{4.21})-(\ref{4.26}), it holds that
	\begin{equation}\label{4.27}
		\left\|u^{*}, q^{*}, B^{*}\right\|_\mathbf{U} \leq \frac{1}{2}\left\|(\tilde{u}_{1}-\tilde{u}_{2},\tilde{q}_{1}-\tilde{q}_{2},\tilde{B}_{1}-\tilde{B}_{2})\right\|_\mathbf{U},
	\end{equation}
	which shows the mapping
	$$
	\Xi: \mathbf{U} \rightarrow \mathbf{U}
	$$
	is contracted. By the Banach fixed point theorem, there exists a unique fixed point $(u, q, B)$ of the mapping $\Xi$ in $\mathbf{U}$. This fixed point satisfies equation (\ref{1.9}) with all terms on the right-hand side expressed in terms of $(u, q, B)$, thereby demonstrating the existence of a global solution. The uniqueness and continuous dependence on the initial data follow naturally from this construction. Thus, we have completed the proof of Theorem \ref{THM1}.

\section{The constant magnetic field outside the fluid}\label{sec5}

For each $t > 0$, $( v, p, H )$ is required to satisfy
the following free boundary problem for the incompressible viscous and resistive MHD equations
\begin{equation}\label{5.1}
	\left\{\begin{aligned}
		&\partial_t v+(v \cdot \nabla) v-\operatorname{div}( \mathbf{D}(v)-P \mathbf{I})=(H \cdot \nabla) H, && \text { in } \Omega(t), \\
		&\partial_t H+(v \cdot \nabla) H=\Delta H+(H \cdot \nabla) v, && \text { in } \Omega(t), \\
		&\operatorname{div} v=0, \quad \operatorname{div} H=0, & & \text { in } \Omega(t),\\
		&(\mathbf{D}(v) -P\mathbf{I}) \mathbf{n}_{t}=0, & & \text { on } \Gamma(t),\\
		&H =\tilde{H}, &&  \text { on } \Gamma(t), \\
		&v|_{t=0}=v_0, \left.H\right|_{t=0}=H_0 && \text { on } \Omega_0.
	\end{aligned}\right.
\end{equation}
Here $\tilde{H}$ is the constant magnetic field outside the fluid and $P$ is  the total pressure.  Set
$$u(t, y) \equiv v(t, x(t, y)) ,\quad q(t, y) \equiv P(t, x(t, y)),  \quad B(t, y) \equiv H(t, x(t, y))-\tilde{H}.
$$
Then the free boundary problem (\ref{5.1}) is equivalent to the following problem for $(u,q,B)$ in
new coordinates:
\begin{equation}\label{5.2}
	\left\{\begin{aligned}
		&\partial_t u-\Delta u+\nabla q=V_{1}(u)+V_{2}(u, q)+\tilde{V}_{3}(u,B),  & &\quad \text{in} \quad \mathbb{R}_{+}\times \mathbb{R}_{+}^N ,
		\\ & \partial_t B-\Delta B=\tilde{V}_{4}(u,B), & & \quad \text{in}\quad  \mathbb{R}_{+}\times \mathbb{R}_{+}^N ,
		\\
		& \operatorname{div} u=V_{5}(u), \quad \operatorname{div} B=V_{6}(u,B), & & \quad \text{in} \quad \mathbb{R}_{+}\times \mathbb{R}_{+}^N ,
		\\
		&\left(\mathbf{D}(u)-q I\right) \mathbf{n}=V_{7}(u)+V_{8}(u, q),& & \quad \text{on}\quad \mathbb{R}_{+}\times \partial \mathbb{R}_{+}^N, \\
		&B=0, & & \quad \text{on}\quad  \mathbb{R}_{+}\times \partial \mathbb{R}_{+}^N,\\
		&u(0, y)=u_0(y), \quad B(0, y)=B_0(y),  & & \quad \text{in}\quad \mathbb{R}_{+}^N,
	\end{aligned}\right.
\end{equation}
where $\mathbf{n} = (0,\dots ,0,-1)$  denotes the outward normal and the nonlinear terms in (\ref{5.2}) are  exactly the same as those in (\ref{1.12}), except for $\tilde{V}_{3}(u,B)$ and $\tilde{V}_{4}(u,B)$, namely
	\begin{align*}
		\tilde{V}_3(u, B)&=(\tilde{H}^{i}+B^{i})\partial_{i}B+(\tilde{H}^{i}+B^{i})P_{i}^{j}\partial_{j}B, \\
		\tilde{V}_4(u, B)&=\sum_{k=1}^N P_k^i \partial_i\left( P_k^j \partial_j B\right)+\sum_{j=1}^NP_{j}^{i}\partial_{i}\partial_{j}B
		+\sum_{i=1}^N \partial_{i}(P_{i}^{j}\partial_{j}B)+ (\tilde{H}^{i}+B^{i})\partial_{i}u+(\tilde{H}^{i}+B^{i})P_{i}^{j}\partial_{j}u.
	\end{align*}

\begin{remark}
In fact, equation (\ref{5.1}) arises from the plasma-vacuum free-interface model, where the plasma is confined within a vacuum containing another magnetic field $\tilde{H}$. A free interface $\Gamma(t)$ moves with the plasma, separating the plasma region $\Omega(t)$ from the vacuum region. $\tilde{H}$ satisfies  in vacuum:
$$
\operatorname{curl} \tilde{H}=0, \quad \operatorname{div} \tilde{H}=0 .
$$

On the interface $\Gamma(t)$, it is required that there is no jump for the pressure or the normal components of magnetic fields:
$$
H \cdot \mathbf{n}_{t}=\tilde{H} \cdot \mathbf{n}_{t}, \quad P:=p+\frac{1}{2}|H|^2=\frac{1}{2}|\tilde{H}|^2.
$$
We consider  a special case where $\tilde{H}$ is a constant vector. Specifically, $\operatorname{curl} \tilde{H}=0$ and  $\operatorname{div} \tilde{H}=0$. Thus, in this paper, equation (\ref{1.3}) corresponds to the scenario where the vacuum magnetic field $\tilde{H}$ vanishes, and the boundary condition $H = 0$ is imposed, which is reasonable for resistive MHD.
\end{remark}

From Sections \ref{sec3} and \ref{sec4}, we only need to estimate $\tilde{V}_3(u, B)$ and $\tilde{V}_4(u, B)$.
By (\ref{A.3}), Proposition \ref{Aprop1} and  employing the same argument as that in proving (\ref{3.1}), we have
\begin{equation}\label{5.3}
	\begin{aligned}
		&\left\| \tilde{V}_3(u, B)\right\|_{L^1(\mathbb{R}_{+} ; \dot{B}_{p, 1}^{-1+\frac{N}{p}}(\mathbb{R}_{+}^N))}\nonumber\\
		&\lesssim \left\| (\tilde{H}^{i}+B^{i})P_{i}^{j}\partial_{j}B\right\|_{L^1(\mathbb{R}_{+} ; \dot{B}_{p, 1}^{-1+\frac{N}{p}}(\mathbb{R}_{+}^N))}+\left\| (\tilde{H}^{i}+B^{i})\partial_{i}B\right\|_{L^1(\mathbb{R}_{+} ; \dot{B}_{p, 1}^{-1+\frac{N}{p}}(\mathbb{R}_{+}^N))}
		\\ & \lesssim  \left(\sum_{r=1}^{N-1}\left\|\int_0^t D u d \tau\right\|_{L^{\infty}(\mathbb{R}_{+} ; \dot{B}_{p, 1}^{\frac{N}{p}}(\mathbb{R}_{+}^N))}^r+1\right)\|(\tilde{H}+B)\nabla B\|_{L^1\left(\mathbb{R}_{+} ; \dot{B}_{p, 1}^{-1+\frac{N}{p}}(\mathbb{R}_{+}^N)\right)}\\&\lesssim \left(\sum_{r=1}^{N-1}\left\|D^2 u\right\|_{L^1(\mathbb{R}_{+} ; \dot{B}_{p, 1}^{-1+\frac{N}{p}}(\mathbb{R}_{+}^N))}^r+1\right)\left(\|\nabla B\|_{L^1(\mathbb{R}_{+} ; \dot{B}_{p, 1}^{\frac{N}{p}}(\mathbb{R}_{+}^N))} \sup _{t>0}\|(\tilde{H}+B)\|_{\dot{B}_{p, 1}^{-1+\frac{N}{p}}(\mathbb{R}_{+}^N)}\right)
		\\&\lesssim \left(\sum_{r=1}^{N-1}\left\|D^2 u\right\|_{L^1(\mathbb{R}_{+} ; \dot{B}_{p, 1}^{-1+\frac{N}{p}}(\mathbb{R}_{+}^N))}^r+1\right)\left(\left\|D^{2} B\right\|_{L^1(\mathbb{R}_{+} ; \dot{B}_{p, 1}^{-1+\frac{N}{p}}(\mathbb{R}_{+}^N))}\left\| \int_t^{\infty} \partial_\tau B d \tau \right\|_{\dot{B}_{p, 1}^{-1+\frac{N}{p}}(\mathbb{R}_{+}^N)}\right)
		\\&\lesssim \left(\sum_{r=1}^{N-1}\left\|D^2 u\right\|_{L^1(\mathbb{R}_{+} ; \dot{B}_{p, 1}^{-1+\frac{N}{p}}(\mathbb{R}_{+}^N))}^r+1\right)\left(\left\|D^{2} B\right\|_{L^1(\mathbb{R}_{+} ; \dot{B}_{p, 1}^{-1+\frac{N}{p}}(\mathbb{R}_{+}^N))}\left\|  \partial_t B  \right\|_{L^{1}(\mathbb{R}_{+} ;\dot{B}_{p, 1}^{-1+\frac{N}{p}}(\mathbb{R}_{+}^N))}\right).
	\end{aligned}
\end{equation}
We have used the fact that $$\left\|f\right\|_{L^1(\mathbb{R}_{+} ; \dot{B}_{p, 1}^{-1+\frac{N}{p}}(\mathbb{R}_{+}^N))}=0,$$
where $f$ is a constant. Thus, employing the same argument as that in proving (\ref{3.13}), we obtain
	\begin{align*}
		\left\| \tilde{V}_4(u, B)\right\|_{L^1(\mathbb{R}_{+} ; \dot{B}_{p, 1}^{-1+\frac{N}{p}}(\mathbb{R}_{+}^N))}&\lesssim \sum_{r=2}^{2N-2}\left\|D^2 u\right\|_{L^1(\mathbb{R}_{+} ; \dot{B}_{p, 1}^{-1+\frac{N}{p}}(\mathbb{R}_{+}^N))}^r \| D^{2} B\|_{L^1(\mathbb{R}_{+} ; \dot{B}_{p, 1}^{-1+\frac{N}{p}}(\mathbb{R}_{+}^N)t)}\\&\quad+\left\| \partial_t B \right\|_{L^{1}(\mathbb{R}_{+},\dot{B}_{p, 1}^{-1+\frac{N}{p}}(\mathbb{R}_{+}^N))}\left(\sum_{r=1}^{N-1}\left\|D^2 u\right\|_{L^1(\mathbb{R}_{+} ; \dot{B}_{p, 1}^{-1+\frac{N}{p}}(\mathbb{R}_{+}^N))}^r+1\right)\\&\quad\times\left( \left\|D^{2} B\right\|_{L^1(\mathbb{R}_{+} ; \dot{B}_{p, 1}^{-1+\frac{N}{p}}(\mathbb{R}_{+}^N))}+\left\|D^{2} u\right\|_{L^1(\mathbb{R}_{+} ; \dot{B}_{p, 1}^{-1+\frac{N}{p}}(\mathbb{R}_{+}^N))}\right).
	\end{align*}

Consequently, the results in Theorem \ref{THM1} remain valid when equations (\ref{1.9}) are replaced by equations (\ref{5.2}).

	\appendix
	
	\section{Notations and useful results}

		\begin{definition}\label{def1}
		$\mathcal{F}[f]$ and $\mathcal{F}_{\xi}^{-1}$ denote the Fourier transform and the inverse Fourier transform, respectively, given by
		$$
		\hat{f}(\xi)=\mathcal{F}[f](\xi)=\int_{\mathbb{R}^N} e^{-i x \cdot \xi} f(x) d x, \quad \mathcal{F}_{\xi}^{-1}[g(\xi)](x)=\frac{1}{(2 \pi)^N} \int_{\mathbb{R}^N} e^{i x \cdot \xi} g(\xi) d \xi.
		$$
	\end{definition}
	
	 Given a nonnegative, monotonic decreasing, smooth radial function  $\hat{\chi}$
	supported in, say, the ball $B(0,2)$ of $\mathbb{R}^{N}$,  and with value 1 on $B(0,1)$, we set
	$$
	\widehat{\varphi}(\cdot):=\widehat{\chi}(\cdot / 2)-\widehat{\chi}(\cdot).
	$$
	We thus have decomposition of unity for $y\in \mathbb{R}^N$, namely
	$$\operatorname{supp} \widehat{\varphi} \subset\left\{y \in \mathbb{R}^N: 1 / 2 \leq|y| \leq 2\right\},\quad \text{and} \quad \sum_{k=-\infty}^{\infty} \widehat{\varphi}\left(2^{-k} y\right)=1 \quad \text{for all} \quad y \in \mathbb{R}^N \backslash\{0\}.
	$$
	
	Then we introduce the homogeneous Littlewood-Paley decomposition $\{\dot{\Delta}_k\}_{k \in \mathbb{Z}}$ over $\mathbb{R}^N$ by setting
	$$
	\dot{\Delta}_k f:=\widehat{\varphi}\left(2^{-k} D\right) f=\mathcal{F}^{-1}\left(\widehat{\varphi}_{k}\left( \cdot\right) \mathcal{F} f\right),
	$$
	where $\mathbb{Z}$ stands for all integers, and $\widehat{\varphi}_k(\cdot)=\widehat{\varphi}\left(2^{-k} \cdot\right)$.
	
	Let us define the homogeneous Besov spaces on $\mathbb{R}^N$. For that, we introduce the following homogeneous Besov semi-norms (for all $s \in \mathbb{R}$ and $p, r \in[1, \infty]$ ):
	\begin{equation}\|f\|_{\dot{B}_{p, r}^s} \equiv\left\{\begin{array}{lr}
			\left(\underset{k \in \mathbb{Z}}\sum 2^{s r k}\|\dot{\Delta}_k f\|_p^r\right)^{1 / r}, & 1 \leq r<\infty, \\
			\underset {k \in \mathbb{Z}}\sup 2^{s k}\|\dot{\Delta}_k f\|_p, & r=\infty.
		\end{array}\right.
	\end{equation}
	
	As $\|f\|_{\dot{\mathrm{B}}_{p, r}^s\left(\mathbb{R}^N\right)}=0$ does not imply that $f \equiv 0$, functions contained in a homogeneous Besov space will be assumed to have the following control on the low Fourier frequencies.  Let $\mathcal{S}_h^{\prime}\left(\mathbb{R}^N\right)$  stand for the set of tempered distributions $f$ over $\mathbb{R}^{N}$ that satisfy
	$$\lim _{\lambda \rightarrow \infty}\|\theta(\lambda D) f\|_{\mathrm{L}^{\infty}\left(\mathbb{R}^N\right)}=0, \quad \text { for any } \theta \in \mathrm{C}_c^{\infty}\left(\mathbb{R}^N\right).
	$$
	Here, $\theta(\lambda D) f$ stands for $\mathcal{F}^{-1} \theta(\lambda \cdot) \mathcal{F} f$ and  the notation $C_c^{\infty}\left(\mathbb{R}^N\right)$ denotes the set of all smooth functions with compact supports. Note that the above condition implies that any distribution in $\mathcal{S}_h^{\prime}\left(\mathbb{R}^N\right)$ tends weakly to 0 at infinity. In particular, $\mathcal{S}_h^{\prime}\left(\mathbb{R}^N\right)$ does not contain nonzero polynomial.

	Now, we define the homogeneous Besov space $\dot{B}_{p, r}^s(\mathbb{R}_{+}^N)$ as the set of all the restriction $f$ of the distribution $\tilde{f} \in \dot{B}_{p, r}^s\left(\mathbb{R}^N\right)$, i.e., $f=\tilde{f}|_{\mathbb{R}_{+}^N}$ with
	$$
	\|f\|_{\dot{B}_{p, r}^s(\mathbb{R}_{+}^N)} \equiv \inf \left\{\|\tilde{f}\|_{\dot{B}_{p, r}^s\left(\mathbb{R}^N\right)}<\infty:  \quad \tilde{f}=\sum_{k \in \mathbb{Z}} \dot{\Delta}_k f \text { in } \mathcal{S}_{h}^{\prime}, \quad f=\tilde{f}|_{\mathbb{R}_{+}^N}\right\}.
	$$
	
		\begin{prop}[cf. \cite{RP}]\label{Aprop2}
		Let $1 \leq p, r<\infty $. Then for  $-1+1 / p<s<1 / p$,  we have
		$$
		\dot{B}_{p, r}^s(\mathbb{R}_{+}^N)=\overline{\left\{f \in \dot{B}_{p, r}^s\left(\mathbb{R}^N\right): \operatorname{supp} f \subset \mathbb{R}_{+}^N\right\}}^{\|\cdot\|_{\dot{B}_{p, r}^s(\mathbb{R}^N)}}.
		$$
	\end{prop}
	
	From Proposition \ref{Aprop2}, it is not difficult to prove that if $1 \leq p, r<\infty $ and $-1+1 / p<s<1 / p$,  then the space $C_{c}^{\infty}(\mathbb{R}_{+}^N)$ is dense in $\dot{B}_{p, r}^s(\mathbb{R}_{+}^N)$.

	We consider the restriction operator $R$ given by
	\begin{equation}\label{A9}
		R f(y)=\left.f(y)\right|_{y \in \mathbb{R}_{+}^N},
	\end{equation}
	for all $f \in \dot{B}_{p, r}^s\left(\mathbb{R}^N\right)$ with $s>0$ and it is understood in the sense of distribution for $s \leq 0$. Let the extension operator $E$ from $\overline{\left\{f \in \dot{B}_{p, r}^s\left(\mathbb{R}^N\right): \operatorname{supp} f \subset \mathbb{R}_{+}^N\right\}}$ be given by the zero-extension, i.e.,
	\begin{equation}\label{A10}
		Ef(y)= \begin{cases}f(y), & \text { in } \mathbb{R}_{+}^N, \\ 0, & \text { in } \overline{\mathbb{R}_{-}^N}.\end{cases}
	\end{equation}
	
	\begin{prop}[cf. \cite{TS}]\label{Aprop3}
		Let $1 \leq p<\infty, 1 \leq r<\infty$ and $-1+1 / p<s<1 / p$, and let $R$ and $E$ be operators defined in (\ref{A9}) and (\ref{A10}). It holds that
		$$
		\begin{aligned}
			& R: \dot{B}_{p, r}^s\left(\mathbb{R}^N\right) \rightarrow \dot{B}_{p, r}^s(\mathbb{R}_{+}^N), \\
			& E: \dot{B}_{p, r}^s(\mathbb{R}_{+}^N) \rightarrow \dot{B}_{p, r}^s\left(\mathbb{R}^N\right),
		\end{aligned}
		$$
		are linear bounded operators. Besides, it holds that
		$$
		R E=I d: \dot{B}_{p, r}^s(\mathbb{R}_{+}^N) \rightarrow \dot{B}_{p, r}^s(\mathbb{R}_{+}^N),
		$$
		where $Id$ denotes the identity operator. Namely, $R$ and $E$ are a retraction and a coretraction, respectively.
	\end{prop}
	
	\begin{remark}
		According to Proposition \ref{Aprop3}, for  $-1+1 / p<s<1 / p$ and $1<p<\infty$,  we can regard that any distribution in $\dot{B}_{p, r}^s(\mathbb{R}_{+}^N)$ under such restriction on $s$ and $p$ can be extended into a distribution over the whole space $\mathbb{R}^N$ and conversely any distribution in $\dot{B}_{p, r}^s\left(\mathbb{R}^N\right)$ is restricted into a distribution over the half-space $\mathbb{R}_{+}^N$.
	\end{remark}
	
		\begin{prop} [cf. \cite{TS}]\label{Aprop1}
		Let $1<p<\infty$ and $-1+1 / p<s \leq N / p-1$. Then for any $f \in \dot{B}_{p, 1}^s(\mathbb{R}_{+}^N)$, it holds
		\begin{equation}\label{A11}
			\|\nabla f\|_{\dot{B}_{p, 1}^s(\mathbb{R}_{+}^N)} \simeq\|f\|_{\dot{B}_{p, 1}^{s+1}(\mathbb{R}_{+}^N)},
		\end{equation}
		where $\simeq$ denotes that the norms on both sides are equivalent.
	\end{prop}
	\begin{prop}[cf. \cite{TS}]\label{Aprop}
		Let $1 \leq p<2 N$, then for all $f \in \dot{B}_{p, 1}^{-1+\frac{N}{p}}(\mathbb{R}_{+}^N)$ and $g \in \dot{B}_{p, 1}^{\frac{N}{p}}(\mathbb{R}_{+}^N)$, there exists a $C>0$ independent of $f$ and $g$ such that the following estimate holds:
		\begin{equation}\label{AA}
			\|f g\|_{\dot{B}_{p, 1}^{-1+\frac{N}{p}}(\mathbb{R}_{+}^N)} \leq C\|f\|_{\dot{B}_{p, 1}^{-1+\frac{N}{p}}(\mathbb{R}_{+}^N)}\|g\|_{\dot{B}_{p, 1}^{\frac{N}{p}}(\mathbb{R}_{+}^N)}.
		\end{equation}
	\end{prop}

	\begin{definition} Let  $X$ be a Banach space with norm $\|\cdot\|_X$. Let $\left\{\phi_k\right\}_{k \in \mathbb{Z}}$ be the Littlewood-Paley dyadic decomposition of unity for $t \in \mathbb{R}$, namely
		$$
		\dot{\Delta}_k f:=\widehat{\phi}\left(2^{-k} D\right) f=\mathcal{F}^{-1}\left(\widehat{\phi}_{k}\left( \cdot\right) \mathcal{F} f\right).
		$$
		For $s \in \mathbb{R}$ and $1 \leq p<\infty$, we define $\dot{F}_{p, r}^s(\mathbb{R};  X)$ as the Bochner-Lizorkin-Triebel space with norm
		$$
		\|\tilde{f}\|_{\dot{F}_{p, r}^s(\mathbb{R};  X)} \equiv\left\{\begin{array}{lr}
			\left\|\left(\underset{k \in \mathbb{Z}}\sum 2^{s r k}\left\|\dot{\Delta}_k(t, \cdot)\right\|_X^r\right)^{1 / r}\right\|_{L^p(\mathbb{R})}, & 1 \leq r<\infty, \\
			\left\|\underset{k \in \mathbb{Z}}\sup  2^{s k}\left\| \dot{\Delta}_k(t, \cdot)\right\|_X\right\|_{L^p(\mathbb{R})}, & r=\infty.
		\end{array}\right.
		$$
	\end{definition}
	Analogously, we define the Bochner-Lizorkin-Triebel spaces $\dot{F}_{p, r}^s( (0,T);  X)$ as the set of all the restriction $f$ of a distribution $\tilde{f} \in \dot{F}_{p, r}^s(\mathbb{R};  X)$ i.e., $f=\tilde{f}|_ {(0,T)}$ on $X$ with
	$$
	\|f\|_{\dot{F}_{p, r}^s((0,T);  X)} \equiv \inf \left\{\|\tilde{f}\|_{\dot{F}_{p, r}^s(\mathbb{R};  X)}<\infty:  \quad f=\tilde{f}|_ {(0,T)}\right\}.
	$$

	For any scalar function $a=a(x)$ and $N$-vector valued function $\mathbf{b}=\left(b_{1}(x), \ldots, b_{N}(x)\right)$, we write
	\begin{align*}
		\nabla a=&\left(\partial_{1} a(x), \ldots, \partial_{N} a(x)\right), \quad \nabla \mathbf{b}=\left(\nabla b_{1}(x), \ldots, \nabla b_{N}(x)\right), \\
		\operatorname{div} \mathbf{b}=&\sum_{j=1}^{N} \partial_{j} b_{j}(x),\quad \nabla^{2} a=\left(\partial_{i} \partial_{j} a\right)_{i, j=1}^N,\quad D^{2} \mathbf{b}=(D^{2}b_1,\ldots,D^{2}b_N).
	\end{align*}
	
	For $y \in  \mathbb{R}^{N} $,  $\langle y\rangle = (1+|y|^{2})^{1/2}$. For any functions $f=f\left(t, y^{\prime}, y_N\right)$ and $g=g\left(t, y^{\prime}, y_N\right), \underset{(t)}{*} g, \underset{\left(t, x^{\prime}\right)}{*} g$ and $\underset{\left(y_N\right)}{*} g$ stand for the convolution between $f$ and $g$ with respect to the variable indicated under $*$, respectively. If both $f$ and $g$ are vector valued functions, $f \underset{\left(t, y^{\prime}\right)}{*} g$ denotes the convolution in $y^{\prime}$ as well as the inner-product of $f$ and $g$, i.e.,
	$$
	f \underset{\left(t, y^{\prime}\right)}{\cdot *} g=\sum_{\ell=1}^{N-1} \int_{\mathbb{R}} \int_{\mathbb{R}^{N-1}} f_{\ell}\left(t-s, y^{\prime}-x^{\prime}\right) g_{\ell}\left(s, x^{\prime}\right) d x^{\prime} d s.
	$$
	
	Denote the norm of $L^p\left(\mathbb{R}^{N-1}\right)$ with variable $y^{\prime} \in \mathbb{R}^{N-1}$  by $\|\cdot\|_{L_{y^{\prime}}^p}$. For the norm on $\dot{F}_{p, r}^s\left(I;  X\left(\mathbb{R}^{N-1}\right)\right)$, we use
	$$
	\|f\|_{\dot{F}_{p, r}^s(I;  X)}=\|f\|_{\dot{F}_{p, r}^s\left(I;  X\left(\mathbb{R}^{N-1}\right)\right)},
	$$
	unless it may cause any confusion. We denote $B(a,b)$ as the open ball centered at $a$ with radius $b>0$. Various constants are simply denoted by $C$ unless otherwise stated.

	We introduce the  Littlewood-Paley decomposition with separation of variables.
	\begin{definition}\label{defa7}
		For $k \in \mathbb{Z}$, let
		\begin{equation}\label{b}
			\begin{aligned}
				& \widehat{\chi}_k\left(y_N\right)=\left\{\begin{array}{ll}
					1, & 0 \leq\left|y_N\right| \leq 2^k, \\
					\text{smooth}, & 2^k<\left|y_N\right|<2^{k+1}, \\
					0, & 2^{k+1} \leq\left|y_N\right|,
				\end{array}\right. \\
				& \widehat{\varphi}_k\left(y_N\right)=\widehat{\chi}_{k+1}\left(y_N\right)-\widehat{\chi}_{k}\left(y_N\right).
			\end{aligned}
		\end{equation}
		and set
		\begin{equation}\label{a}
			\widehat{\Phi}_k(y) \equiv \widehat{\varphi}_k\left(\left|y^{\prime}\right|\right) \otimes \widehat{\chi}_{k-1}\left(y_N\right)+\widehat{\chi}_k\left(\left|y^{\prime}\right|\right) \otimes \widehat{\varphi}_k\left(y_N\right).
		\end{equation}
	\end{definition}
	Then it is obvious  that
	$$
	\sum_{k \in \mathbb{Z}} \widehat{\Phi}_k(y) \equiv 1, \quad y=\left(y^{\prime}, y_N\right) \in \mathbb{R}^N \backslash\{0\}.
	$$
	
	\begin{remark}[Varieties of the Littlewood-Paley dyadic decompositions]
		Let $\left(\tau, y^{\prime}, y_N\right) \in \mathbb{R} \times \mathbb{R}^{N-1} \times \mathbb{R}$ be Fourier's adjoint variables corresponding to $\left(t, x^{\prime}, \eta\right) \in \mathbb{R} \times \mathbb{R}^{N-1} \times \mathbb{R}$.
\begin{enumerate}[1.]
	\item $\left\{\Phi_k(x)\right\}_{k \in \mathbb{Z}}$ : the Littlewood-Paley dyadic decomposition over $x=\left(x^{\prime}, \eta\right) \in \mathbb{R}^N$ given by (\ref{a}).
	
	\item $\left\{\phi_k(t)\right\}_{k \in \mathbb{Z}}$ : the Littlewood-Paley dyadic decompositions in $t \in \mathbb{R}$.
	
	\item $\left\{\varphi_k\left(x^{\prime}\right)\right\}_{k \in \mathbb{Z}}$ and $\left\{\phi_k(\eta)\right\}_{k \in \mathbb{Z}}$ : the standard (annulus type) Littlewood-Paley dyadic decompositions in $x^{\prime} \in \mathbb{R}^{N-1}$ and $\eta \in \mathbb{R}$, respectively.
	
	\item $\left\{\chi_k\left(x^{\prime}\right)\right\}_{k \in \mathbb{Z}}$ and $\left\{\chi_k(\eta)\right\}_{k \in \mathbb{Z}}$ : the lower frequency smooth cut-off function given by (\ref{b}), respectively.
	
	\item All the above-defined decompositions are even functions.
\end{enumerate}		
	\end{remark}

	\begin{theorem}\label{thmA1}
		Let $1 <p < \infty$ and $-1+1/p<s$.  For all function $f=f\left(t, y^{\prime}, \eta\right) \in C\left(\overline{\mathbb{R}_{+}};  \dot{B}_{p, 1}^s(\mathbb{R}_{+}^N)\right) \cap\dot{W}^{1,1}\left(\mathbb{R}_{+};  \dot{B}_{p, 1}^s(\mathbb{R}_{+}^N)\right)$, $D^{2} f\left(t, y^{\prime}, \eta\right) \in L^1\left(\mathbb{R}_{+};  \dot{B}_{p, 1}^s(\mathbb{R}_{+}^N)\right)$,   the following estimate holds:
		\begin{equation}\label{A.1}
			\begin{aligned}
				& \sup _{\eta \in \mathbb{R}_{+}}\left(\|f(\cdot, \cdot, \eta)\|_{\dot{F}_{1,1}^{1-1 / 2 p}(\mathbb{R}_{+};  \dot{B}_{p, 1}^s\left(\mathbb{R}^{N-1}\right))}+\|f(\cdot, \cdot, \eta)\|_{L^1(\mathbb{R}_{+};  \dot{B}_{p, 1}^{s+2-1 / p}\left(\mathbb{R}^{N-1}\right))}\right) \\
				& \lesssim \left\|\partial_t f\right\|_{L^1(\mathbb{R}_{+};  \dot{B}_{p, 1}^s(\mathbb{R}_{+}^N))}+\left\|D^2 f\right\|_{L^1(\mathbb{R}_{+};  \dot{B}_{p, 1}^s(\mathbb{R}_{+}^N))}.
			\end{aligned}
		\end{equation}
	\end{theorem}
	
	\begin{proof}
		By the definition of $\dot{B}_{p, 1}^s(\mathbb{R}_{+}^N)$, for any $\epsilon>0$, there exists $$\widetilde{f} \in C\left(\overline{\mathbb{R}_{+}};  \dot{B}_{p, 1}^s(\mathbb{R}^N)\right) \cap \dot{W}^{1,1}\left(\mathbb{R}_{+};  \dot{B}_{p, 1}^s(\mathbb{R}^N)\right) \cap L^1\left(\mathbb{R}_{+};  \dot{B}_{p, 1}^{s+2}(\mathbb{R}^N)\right),$$
		such that
		\begin{align*}
			& \|\widetilde{f}\|_{L^{\infty}\left(\mathbb{R}_{+};  \dot{B}_{p, 1}^s\left(\mathbb{R}^N\right)\right) \cap \dot{W}^{1,1}\left(\mathbb{R};  \dot{B}_{p, 1}^s\left(\mathbb{R}^N\right)\right)} \leq\|f\|_{L^{\infty}(\mathbb{R}_{+};  \dot{B}_{p, 1}^s(\mathbb{R}_{+}^N)) \cap \dot{W}^{1,1}\left(\mathbb{R}_{+};  \dot{B}_{p, 1}^s(\mathbb{R}_{+}^N)\right)}+\epsilon, \\
			& \|D^2 \widetilde{f}\|_{L^1(\mathbb{R}_{+};  \dot{B}_{p, 1}^s(\mathbb{R}^N))} \leq\left\|D^2 f\right\|_{L^1\left(\mathbb{R}_{+};  \dot{B}_{p, 1}^s(\mathbb{R}_{+}^N)\right)}+\epsilon.
		\end{align*}
		
		We then extend $\tilde{f}$ into $t<0$ by an even extension. For simplicity, we denote $\tilde{f}$ as $f$ in the following. It directly follows that
		\begin{equation}\label{A.2}
			f \in L^2\left(\mathbb{R}; \dot{B}_{p, 1}^{s+1}\left(\mathbb{R}^N\right)\right).
		\end{equation}
		From (\ref{A.2}), we see that
		\begin{equation}\label{A.3}
			\lim_{t\rightarrow\pm \infty}f(t, y^{\prime}, \eta) = 0, \quad \text{for a.e. }  (y^{\prime}, \eta) \in\mathbb{R}^{N-1}\times \mathbb{R}.
		\end{equation}
		As $\widehat{f}_{j}(\cdot)=\widehat{f}(2^{-j}\cdot)$, we have 
		 $$f_{j}(\cdot)=2^{jN}f(2^{j}\cdot).$$

		For $1 < p <\infty$, it holds
			\begin{align}\label{A.4}
				 &\| f(\cdot, \cdot, \eta) \|_{\dot{F}_{1,1}^{1-1 / 2 p}(\mathbb{R}_{+};  \dot{B}_{p, 1}^s(\mathbb{R}_{y^{\prime}}^{N-1}))}\nonumber\\
				\leq & \left\|\sum_{i \in \mathbb{Z}} \sum_{j \geq 2 i} 2^{(1-1 / 2 p) j} 2^{s i}\left\|\phi_j \underset{(t)}{*} \varphi_i \underset{\left(y^{\prime}\right)}{*} f(t, \cdot, \eta)\right\|_{L^p(\mathbb{R}_{y^{\prime}}^{N-1})}\right\|_{L_t^1\left(\mathbb{R}_{+}\right)} \nonumber \\
				& +\left\|\sum_{i \in \mathbb{Z}} \sum_{j \leq 2 i} 2^{(1-1 / 2 p) j} 2^{s i}\left\| \phi_j \underset{(t)}{*} \varphi_i \underset{\left(y^{\prime}\right)}{*} f(t, \cdot, \eta)\right\|_{L^p(\mathbb{R}_{y^{\prime}}^{N-1})}\right\|_{L_t^1(\mathbb{R}_{+})} =: I+II.
			\end{align}
		Using (\ref{A.3}), we have
			\begin{align*}
				\phi_j(t) \underset{(t)}{*}  f\left(t,\cdot, \eta\right)&  =-\int_{\mathbb{R}_{+}} \partial_\tau\left(\int_\tau^{\infty} \phi_j(t-r) d r\right)  f\left(\tau, \cdot, \eta\right) d \tau \\
				& =-\left[\left(\int_\tau^{\infty} \phi_j(t-r) d r\right)  f\left(\tau, \cdot, \eta\right)\right]_{\tau=0}^{\infty}+\int_{\mathbb{R}_{+}}\left(\int_\tau^{\infty} \phi_j(t-r) d r\right) \partial_\tau \left(\tau, \cdot, \eta\right) d \tau \\
				& =\partial_t^{-1} \phi_j(t) \underset{(t)}{*} \partial_t  f\left(t, \cdot, \eta\right),
			\end{align*}
		where we set
		$$
		\partial_t^{-1} \phi_j(t-\tau) \equiv \int_{-\infty}^{t-\tau} \phi_j(r) \mathrm{d} r=\int_\tau^{\infty} \phi_j(t-r) d r.
		$$
		Here, we recall the Littlewood-Paley decomposition in Definition \ref{defa7}. Since $\partial_t^{-1} \phi_j$ is also a rapidly decreasing smooth function, by using the Hausdorff-Young inequality and $\chi_{j}(\eta)\ast \chi_{j-1}(\eta)=\chi_{j-1}(\eta)$, we have
		\begin{align}
			& I=\left\|\sum_{i \in \mathbb{Z}} \sum_{j \geq 2 i} 2^{(1-1 / 2 p) j} 2^{s i}\left\| \partial_t^{-1} \phi_j \underset{(t)}{*} \varphi_i\underset{\left(y^{\prime}\right)}{*} \partial_t u(t, \cdot, \eta)\right\|_{L^p(\mathbb{R}_{y^{\prime}}^{N-1})}\right\|_{L_t^1(\mathbb{R}_{+})}\notag \\
			& \lesssim\left\|\sum_{i \in \mathbb{Z}} 2^{s i} \sum_{j \geq 2 i} 2^{(1-/ 2 p) j}\left|2^{-j}\left(\partial_t^{-1} \phi_{0}\right)_j(t)\right| \underset{(t)}{*}\left\|\varphi_i \underset{\left(y^{\prime}\right)}{*} \partial_t f(t, \cdot, \eta)\right\|_{L^p(\mathbb{R}_{y^{\prime}}^{N-1})}\right\|_{L_t^1(\mathbb{R}_{+})} \notag\\
			& \lesssim \sum_{i \in \mathbb{Z}} 2^{s i} \sum_{j \geq 2 i} 2^{-\frac{1}{2 p} i}\left\|\int_{\mathbb{R}_{+}} \frac{2^i}{\left\langle 2^i(t-\tau)\right\rangle^2}\left\| \varphi_i \underset{\left(y^{\prime}\right)}{*} \partial_\tau f(\tau, \cdot, \eta)\right\|_{L^p(\mathbb{R}_{y^{\prime}}^{N-1})} d \tau\right\|_{L_t^1(\mathbb{R}_{+})} \notag\\& \lesssim \sum_{i \in \mathbb{Z}} 2^{s i} \sum_{j \geq 2 i} 2^{-\frac{1}{2 p} j}\left\|\frac{2^j}{\left\langle 2^j t\right\rangle^2}\right\|_{L_t^1\left(\mathbb{R}_{+}\right)}\left\|\left\| \sum_{m \in \mathbb{Z}} \Phi_m \underset{\left(y^{\prime}, \eta\right)}{*} \varphi_i \underset{\left(y^{\prime}\right)}{*} \partial_t f(t, \cdot, \eta)\right\|_{L^p(\mathbb{R}_{y^{\prime}}^{N-1})}\right\|_{L_t^1(\mathbb{R}_{+})}\notag \\
			& \lesssim \sum_{i \in \mathbb{Z}} 2^{s i}\left\|\sum_{j \geq 2 i} 2^{-\frac{1}{2 p} j}\left\| \sum_{|m-i| \leq 1} \Phi_m \underset{\left(y^{\prime}, \eta\right)}{*} \varphi_i \underset{\left(y^{\prime}\right)}{*} \chi_i(\eta) \underset{(\eta)}{*} \chi_{i-1}(\eta) \underset{(\eta)}{*} \partial_t f(t, \cdot, \eta)\right\|_{L^p(\mathbb{R}_{y^{\prime}}^{N-1})}\right\|_{L_t^1(\mathbb{R}_{+})} \notag\\
			& \lesssim  \sum_{i \in \mathbb{Z}} 2^{s i}\left\|\sum_{j \geq 2 i} 2^{-\frac{1}{2 p} j}\left\| \chi_i(\eta)\right\|_{L^{p^{\prime}}(\mathbb{R}_\eta^{+})}\left\|\left\|\Phi_i \underset{\left(y^{\prime}, \eta\right)}{*} \partial_t f(t, \cdot, \eta)\right\|_{L^p(\mathbb{R}_{y^{\prime}}^{N-1})}\right\|_{L^p(\mathbb{R}_\eta^{+})}\right\|_{L_t^1(\mathbb{R}_{+}t)} \notag\\
			& \lesssim \left\|\sum_{i \in \mathbb{Z}} 2^{s i}\left\| \Phi_i \underset{\left(y^{\prime}, \eta\right)}{*} \partial_t f(t, \cdot, \eta)\right\|_{L^p(\mathbb{R}_{+,(y^{\prime}, \eta)}^N)}\right\|_{L_t^1(\mathbb{R}_{+})}\nonumber\\
			& \lesssim \left\|\partial_t f\right\|_{L^1(\mathbb{R}_{+};  \dot{B}_{p, 1}^s(\mathbb{R}_{+}^N))}. \label{A.6}
		\end{align}
		On the other hand, by the Minkowski inequality and $p>1$, we have
		\begin{align}
			I I & =\left\|\sum_{i \in \mathbb{Z}} \sum_{j \leq 2 i} 2^{(1-1 / 2 p) j} 2^{s i}\left\| \phi_j \underset{(t)}{*} \varphi_j \underset{\left(y^{\prime}\right)}{*} f(t, \cdot, \eta)\right\|_{L^p(\mathbb{R}_{y^{\prime}}^{N-1})}\right\|_{L_t^1(\mathbb{R}_{+})} \notag\\
			& \lesssim \sum_{i \in \mathbb{Z}} 2^{s i} \sum_{j \leq 2 i} 2^{(1-1 / 2 p) j}\left\|\left|\phi_j\right| \underset{(t)}{*}\left\| \varphi_i \underset{\left(y^{\prime}\right)}{*} f(t, \cdot,\eta)\right\|_{L^p(\mathbb{R}_{y^{\prime}}^{N-1})}\right\|_{L_t^1(\mathbb{R}_{+})} \notag\\
			& \lesssim \sum_{i \in \mathbb{Z}} 2^{s i} \sum_{j \leq 2 i} 2^{(1-1 / 2 p) j}\left\|\left\| \sum_{m \in \mathbb{Z}} \Phi_m \underset{\left(y^{\prime}, \eta\right)}{*} \chi_i(\eta) \underset{(\eta)}{*} \chi_{i-1}(\eta) \underset{(\eta)}{*} \varphi_i \underset{\left(y^{\prime}\right)}{*} f(t, \cdot, \eta)\right\|_{L^p(\mathbb{R}_{y^{\prime}}^{N-1})}\right\|_{L_t^1(\mathbb{R}_{+})} \notag\\
			& \lesssim \left\|\sum_{i \in \mathbb{Z}} 2^{s i} 2^{(2-1 / p) i}\left\| \sum_{|m-i| \leq 1} \Phi_m \underset{\left(y^{\prime}, \eta\right)}{*} \chi_i(\eta) \underset{(\eta)}{*} \varphi_i \underset{\left(y^{\prime}\right)}{*} f(t, \cdot, \eta)\right\|_{L^p(\mathbb{R}_{y^{\prime}}^{N-1})}\right\|_{L_t^1(\mathbb{R}_{+})} \notag\\
			& \lesssim \left\|\sum_{i \in \mathbb{Z}} 2^{(s+2-1 / p) i}\left\| \chi_i\right\|_{L^{p^{\prime}}(\mathbb{R}_{+,, \eta})}\left\| \Phi_i \underset{\left(y^{\prime}, \eta\right)}{*} f(t, \cdot, \eta)\right\|_{L^p(\mathbb{R}_{+}^N)}\right\|_{L_t^1(\mathbb{R}_{+})}
			\notag\\
			&
			\lesssim \left\|\sum_{i \in \mathbb{Z}} 2^{(s+2) i}\left\|(D^{2})^{-1} \Phi_i \underset{\left(y^{\prime}, \eta\right)}{*} (D^{2})^{-1} f(t, \cdot, \eta)\right\|_{L^p(\mathbb{R}_{+}^N)}\right\|_{L_t^1\left(\mathbb{R}_{+}\right)} \nonumber \\
			& \lesssim \left\|\sum_{i \in \mathbb{Z}} 2^{s i}\left\| \Phi_i \underset{\left(y^{\prime}, \eta\right)}{*} D^{2} f(t, \cdot, \eta)\right\|_{L^p(\mathbb{R}_{+}^N)}\right\|_{L_t^1(\mathbb{R}_{+})}
			\nonumber\\
			&=C\|D^{2} f\|_{L^1(\mathbb{R}_{+};  \dot{B}_{p, 1}^s(\mathbb{R}_{+}^N))}. \label{A.7}
		\end{align}
		
		For $1 < p <\infty$, we can apply a similar treatment for the spatial variables. Indeed, we have
			\begin{align}\label{A.8}
				& \|f(\cdot, \cdot, \eta)\|_{L^1(\mathbb{R}_{+};  \dot{B}_{p, 1}^{s+2-1 / p}(\mathbb{R}_{y^{\prime}}^{N-1}))} \nonumber\\
				& \lesssim \left\|\sum_{i \in \mathbb{Z}} 2^{s i} 2^{(2-1 / p) i}\left\| \varphi_i \underset{\left(y^{\prime}\right)}{*} \chi_{i-1}(\eta) \underset{(\eta)}{*} \sum_{|m-i| \leq 1} \Phi_m\left(y^{\prime}, \eta\right) \underset{\left(y^{\prime}, \eta\right)}{*} f(t, \cdot, \eta)\right\|_{L^p(\mathbb{R}_{y^{\prime}}^{N-1})}\right\|_{L_t^1(\mathbb{R}_{+})}\nonumber \\
				& \lesssim \left\|\sum_{i \in \mathbb{Z}} 2^{(s+2) i}\left(2^{-i}\left\|\varphi_i \underset{\left(y^{\prime}\right)}{*} \chi_{i-1}(\eta) \underset{(\eta)}{*} \chi_i(\eta) \underset{(\eta)}{*} \Phi_i\left(y^{\prime}, \eta\right) \underset{\left(y^{\prime}, \eta\right)}{*} f(t, \cdot, \eta)\right\|_{L^p(\mathbb{R}_{y^{\prime}}^{N-1})}^p\right)^{1 / p}\right\|_{L_t^1(\mathbb{R}_{+})}\nonumber
				\\
				& \lesssim \left\|\sum_{i \in \mathbb{Z}} 2^{(s+2) i}\left(2^{-i} 2^i\left\|\varphi_i \underset{\left(y^{\prime}\right)}{*} \Phi_i\left(y^{\prime}, \eta\right) \underset{\left(y^{\prime}, \eta\right)}{*} f(t, \cdot, \eta)\right\|_{L^p(\mathbb{R}_{+}^N)}^p\right)^{1 / p}\right\|_{L_t^1(\mathbb{R}_{+})} \nonumber\\
				& \lesssim \left\|\sum_{i \in \mathbb{Z}} 2^{(s+2) i}\right\| \Phi_i\left(y^{\prime}, \eta\right) \underset{\left(y^{\prime}, \eta\right)}{*} f(t, \cdot, \eta)\left\|_{L^p(\mathbb{R}_{y^{\prime}}^{N-1} \times \mathbb{R}_\eta^{+})}\right\|_{L_t^1(\mathbb{R}_{+})}\nonumber \\
				& \lesssim \left\|\sum_{i \in \mathbb{Z}} 2^{s i}\right\| \Phi_i(\cdot, \cdot) \underset{\left(y^{\prime}, \eta\right)}{*} D^{2}f(s, \cdot, \cdot)\left\|_{L^p(\mathbb{R}_{y^{\prime}}^{N-1} \times \mathbb{R}_\eta^{+})}\right\|_{L_t^1(\mathbb{R}_{+})}\nonumber\\ &\lesssim \int_{\mathbb{R}_{+}}\|D^{2} f(t, y)\|_{\dot{B}_{p, 1}^s} d t.
			\end{align}

		Combining estimates (\ref{A.4})-(\ref{A.8}), we conclude  estimate (\ref{A.1}).
	\end{proof}

	\section*{Acknowledgments}
	
	Hao, Yang and Zhang were partially supported by NSF of China under Grant No. 12171460. Hao was also partially supported by CAS Project for Young Scientists in Basic Research under Grant No. YSBR-031 and National Key R\&D Program of China under Grant No. 2021YFA1000800.

\end{document}